\documentclass[12pt]{article}
\usepackage{amsmath,amssymb,amsthm
,makeidx,fancyhdr
}
\usepackage{color}
\usepackage[all]{xy}
\usepackage[latin1]{inputenc}
\usepackage{tikz}
\usetikzlibrary{shapes,arrows}

\theoremstyle{definition}
\newtheorem{theorem}{Theorem}[section]
\newtheorem{proposition}[theorem]{Proposition}

\newtheorem{claim}[theorem]{Claim}

\newtheorem{lemma}[theorem]{Lemma}
\newtheorem{corollary}[theorem]{Corollary}

\newtheorem{remark}[theorem]{Remark}
\newtheorem{remarks}[theorem]{Remarks}
\newtheorem{notation}[theorem]{Notations}
\newtheorem{definition}[theorem]{Definition}

\newtheorem{fact}[theorem]{Fact}

\newcommand{\blue}{\textcolor[rgb]{0.00,0.00,1.00}}

\newcommand{\mo}{\triangleleft}
\newcommand{\Ult}{\operatorname{Ult}}

\newcommand{\fr}{{}^\frown}

\newcommand{\name}{\dot}
\newcommand{\can}{\check}
\newcommand{\force}{\Vdash}

\newcommand{\la}{\langle}
\newcommand{\ra}{\rangle}

\newcommand{\uhr}{\upharpoonright}

\newcommand{\po}{\mathcal{P}}

\newcommand{\D}{\mathfrak{D}}

\newcommand{\delset}{\Delta}

\newcommand{\U}{\mathcal{U}}

\newcommand{\K}{\mathcal{K}}

\newcommand{\pzero}{\mathcal{P}^0}

\renewcommand{\i}{i}

\newcommand{\W}{\mathcal{W}}
\newcommand{\rsb}{\supseteq} 

\newcommand{\pone}{\mathcal{P}^1}

\newcommand{\qzero}{\mathcal{Q}^0}
\newcommand{\qone}{\mathcal{Q}^1}

\newcommand{\Gzero}{G^0}
\newcommand{\Gone}{G^1}

\newcommand{\Uz}[2]{U^0_{(#1,#2)}}
\newcommand{\Uo}[2]{U^1_{(#1,#2)}}

\newcommand{\jz}[2]{j^0_{(#1,#2)}}
\newcommand{\iz}[4]{i^{0,M_{(#1,#2)}}_{(#3,#4)}}
\newcommand{\jdz}[4]{j^0_{(#1,#2),(#3,#4)}}
\newcommand{\jo}[2]{j^1_{(#1,#2)}}
\newcommand{\Mz}[2]{M^0_{(#1,#2)}}
\newcommand{\Mo}[2]{M^1_{(#1,#2)}}

\newcommand{\T}{T}

\newcommand{\X}[2]{X_{(#1,#2)}}
\newcommand{\pX}{\mathcal{P}^X}
\newcommand{\qX}{\mathcal{Q}^X}
\newcommand{\GX}{G^X}
\newcommand{\UX}[2]{U^X_{(#1,#2)}}
\newcommand{\jX}[2]{j^X_{(#1,#2)}}
\newcommand{\MX}[2]{M^X_{(#1,#2)}}

\DeclareMathOperator{\otp}{otp}
\DeclareMathOperator{\Sacks}{Sacks}
\DeclareMathOperator{\Code}{Code}

\DeclareMathOperator{\supp}{supp}
\DeclareMathOperator{\On}{On}
\DeclareMathOperator{\cp}{cp}
\DeclareMathOperator{\len}{len}
\DeclareMathOperator{\Cf}{Cof}
\DeclareMathOperator{\rng}{rng}
\DeclareMathOperator{\dom}{dom}

\DeclareMathOperator{\rank}{rank}

\DeclareMathOperator{\GCH}{GCH}
\DeclareMathOperator{\CU}{CU}
\DeclareMathOperator{\cu}{cu}
\DeclareMathOperator{\tamerank}{Trank}

\title{The structure of the Mitchell order - I}
\author{Omer Ben-Neria 
\footnote{The paper is a part of the author Ph.D. written in Tel-Aviv University under the supervision of
Professor Moti Gitik.}}
\date{{}}

\begin{document}
\maketitle
\begin{abstract}
We isolate here a wide class of well founded orders called tame orders, 
and show that each such order of cardinality at most $\kappa$ can be realized as the Mitchell order on 
a measurable cardinal $\kappa$, from a consistency assumption weaker than $o(\kappa) = \kappa^+$.
\end{abstract}

\section{Introduction}\label{Section - Introduction}
This paper is the first of {a two-part study} on 
the possible structure of the Mitchell order.
In this first paper{,} we identify a large class of well-founded
orders with some appealing properties, and prove
that each of its members can be realized as {$\mo(\kappa)$ -- the Mitchell order
on the set of normal measures on $\kappa$.} 
In \cite{Mitchell - MO} Mitchell introduced the following relation:
Given two normal measures $U,W$, we write $U \mo W$ to denote that
$U \in M_W \cong \Ult(V,W)$. 
Mitchell proved that $\mo$ is a well founded order now known as the Mitchell 
ordering. The Mitchell ordering and its extension to arbitrary extenders have become a major tool in the study of large cardinals, with important applications to
consistency results and inner model theory.
Given a cardinal $\kappa$, we write $o(\kappa)$ to denote the rank of the well-founded order $\mo(\kappa)$.
The research on the possible structure on the Mitchell order $\mo(\kappa)$ is closely related to the question of its possible size, namely, the number of normal measures on $\kappa$: The first results by Kunen \cite{Kunen} and by Kunen and Paris \cite{Kunen-Paris} showed that this number can take the extremal values of $1$ and $\kappa^{++}$ (in a model of $\GCH$) respectively.
Soon after, Mitchell \cite{Mitchell - MO} \cite{Mitchell - MO-rev}, showed that this size can be any cardinal $\lambda$ between $1$ and $\kappa^{++}$, under the large cardinal assumption and in a model of $o(\kappa) = \lambda$. Baldwin \cite{Baldwin} showed that for $\lambda < \kappa$ and from stronger large cardinal assumptions, $\kappa$ can also be the first measurable cardinal.
Apter-Cummings-Hamkins \cite{Apter Cummings Hamkins} proved that there can be $\kappa^+$ normal measures on $\kappa$ from the minimal assumption of a single measurable cardinal; for $\lambda < \kappa^+$, Leaning \cite{Leaning} reduced the large cardinal assumption from $o(\kappa) = \lambda$ to an assumption weaker than $o(\kappa) = 2$. 
The question of the possible number of normal measures on $\kappa$ was finally resolved by Friedman and Magidor in \cite{Friedman-Magidor - Normal Measures}, were it is shown that $\kappa$ can carry any number of normal measures $1 \leq \lambda \leq \kappa^{++}$ from the minimal assumption.
The Friedman-Magidor poset will be extensively used in this paper and {the subsequent} part II.

Further results where obtained on the possible structure of the Mitchell order: Mitchell \cite{Mitchell - MO} and Baldwin \cite{Baldwin} showed that 
from some large cardinal assumptions, every well-order and pre-well-order (respectively) can be isomorphic to $\mo(\kappa)$ at some $\kappa$. 
Cummings \cite{Cummings I},\cite{Cummings II}, and Witzany \cite{Witzany} studied the $\mo$ ordering in various generic extensions, and showed that $\mo(\kappa)$ can have a rich structure.  Cummings constructed models where $\mo(\kappa)$ embeds every order from a specific family of orders we call tame. Witzany showed that in a Kunen-Paris extension of a Mitchell model $L[\mathcal{U}]$, with $o^{\U}(\kappa) = \kappa^{++}$, every well-founded order of cardinality $\leq \kappa^+$ embeds into $\mo(\kappa)$. However, the general question of the possible structure of $\mo(\kappa)$ {has} remained open.

In this paper and {the subsequent} part II (\cite{OBN - MOII}) we gradually develop a series of techniques by which 
we obtain increasing variety of possible $\mo$ structures from increasingly large cardinal assumptions:
In this paper, we develop a technique {for realizing} a wide family of well founded orders 
 called \emph{tame orders} from assumptions weaker than the existence of a measurable cardinal $\kappa$ with $o(\kappa) = \kappa^+$. 
 In part II (\cite{OBN - MOII}) we increase our large cardinal assumption slightly above the existence of a sharp to a strong cardinal $0^{\P}$,
 and show that every well-founded order can be consistently realized as $\mo(\kappa)$ on {a} measurable cardinal $\kappa$. \\
{The forcing constructions in both papers obey the following guidelines:}\\

\noindent \textbf{1.} The ground model $V = \K(V)$ is a core model, presented as an extender model $L[E]$ (as in \cite{Steel - HB}) or a Mitchell model $L[\U]$ (see \cite{Mitchell - MO-rev}
or \cite{Mitchell - HB}).\\

\noindent \textbf{2.} An intermediate forcing extension $V' = V[G']$ is introduced, to serve as an intermediate ground for a final $\mo$ structure. 
{Our goal is to make $\mo(\kappa)^{V'}$  as rich as possible (relative to the large cardinal assumption) while ensuring that
the normal measures on $\kappa$ are separated by sets.} We say that the normal measures on $\kappa$ are separated 
when one can assign each normal measure $U$ on $\kappa$ a set $X_U \in U$ which does not belong to any distinct normal measure $U' \neq U$.\\

\noindent \textbf{3.} In a final last extension, we restrict $\mo(\kappa)^{V'}$ to any chosen $\mathcal{W} \subset \mo(\kappa)^{V'}$ of 
cardinality $|\mathcal{W}| \leq \kappa$. We refer to this last forcing as a \emph{final cut}.
The final cut relies on the fact that the normal measures in $V'$ are separated by sets.\\

\noindent The orders at the center of this paper are \emph{tame orders}. For every ordinal $\lambda$, we define an order 
$(R_\lambda,<_{R_{\lambda}})$ by $R_\lambda = \{(\alpha,\beta) \in \lambda^2 \mid \alpha \leq \beta\}$,  and $(\alpha,\beta) <_{R_\lambda} (\alpha',\beta') \iff \beta < \alpha'$.
In Section \ref{Section - I - Tame Orderings}, we introduce tame orders 
and show that up to a simple operation called \emph{reduction}, every tame order $(S,<_S)$ embeds in some $(R_\lambda,<_{R_{\lambda}})$. The rest of the paper is {largely devoted} to realizing $R_\lambda$ using $\mo(\kappa)$. 
The following observation relates $R_{\lambda}$ to the generalized Mitchell order in $V$: Suppose that $\la U_\alpha \mid \alpha < \lambda\ra$ is a $\mo-$increasing sequence, and let $t$ be the map defined by 
\[
t(\alpha,\beta) = 
\begin{cases}
 U_\alpha \times U_\beta & \text{ if } \alpha<\beta \\
  U_\alpha & \text{ if } \alpha =\beta. \\
\end{cases}
\]
Then $t$ defines an isomorphism of $(R_\lambda,<_{R_\lambda})$ with a set of ultrafilters in $V$, ordered
by $\mo$. 
Here $U_{\alpha} \times U_\beta = \{ X \subseteq \kappa^2 \mid \{\nu \mid \{\mu \mid (\nu,\mu) \in X\} \in U_\beta\}\in U_\alpha\}$.
The purpose of the main forcing is to reduce each $U_\alpha \times U_\beta$ to a normal measure on $\kappa$ to construct an intermediate model $V'$ where $\mo(\kappa)^{V'}$ embeds $R_\lambda$. 
This is done by forcing with a Magidor iteration of one-point Prikry forcings  $\pone$. The iteration introduces an almost injective function
$d : \kappa \to \kappa$ and ultrafilters $\Uo{\alpha}{\beta}$, so that the map $\nu \mapsto (\nu,d^{-1}(\nu))$ defines an isomorphism of $\Uo{\alpha}{\beta}$ with an extension of $U_\alpha\times U_\beta$.
However, there is a problem with forcing directly over $V$. 
The one-point Prikry forcing at stage $\nu < \kappa$, is based on a normal measure $U_{\nu,\alpha}$ on $\nu$ (where $\alpha < o(\nu)$).
For each $\beta < o(\kappa)$, we get that the choice $\nu \mapsto U_{\nu,\alpha}$ determines a unique measure $U_{\alpha}$ modulo $U_\beta$ ($\alpha < \beta$), and it follows that we cannot form a normal projection of $U_\alpha \times U_\beta$ for more than a single $\alpha$. This problem is solved by first forcing with a Friedman-Magidor poset $\pzero$. The forcing $\pzero$ splits each $U_\beta$ into $\kappa$ many $\mo-$equivalent extensions. The different extensions allow us to simultaneously deal with $U_\alpha \times U_\alpha$ for every $\alpha < \beta$. 

{Sections \ref{Section - I - description of j}} through {\ref{Section - I - identify normal measures}} provide an analysis of the normal measures $\Uo{\alpha}{\beta}$ in a $\pzero * \pone$ generic extension $V'$. 
It is shown that $\mo(\kappa)^{V'}$ embeds $R_\lambda$ and that the normal measures are separated by sets. 
The analysis of the measures $\Uo{\alpha}{\beta}$ focuses on the iterated ultrapowers obtained from the restriction of $i_{\Uo{\alpha}{\beta}} : V' \to \Ult(V',\Uo{\alpha}{\beta})$ to $V = \K(V')$.
Finally, in {Section \ref{Section - I - final cut}}, we introduce the \emph{final cut} iteration and apply it to $V'$ to remove unwanted measures without changing the $\mo$ structure on the rest. This construction is used to prove the main result in this paper, Theorem \ref{Theorem - I - main theorem}.\\

\noindent {The notations in this paper obey the following conventions:}\\
A pair $(S,<_S)$ will be called an \emph{order}, if 
$<_S \subset S \times S$ is a relation which defines a partial order
(anti-symmetric and transitive relation) on $S$. 
When there is no danger of confusion, we will use $S$
to denote the entire order $(S,<_S)$.
By a \emph{suborder} of $(S,<_S)$ we mean the restriction of $(S,<_S)$ to a subset $X \subset S$, and denote it by $(X,<_S\uhr X)$.
We use the Jerusalem convention for the forcing order, in which $p \geq q$ means that
$p$ is stronger than $q$. Thus the trivial condition of $\po$ will be denoted by $0_{\po}$.
A name of a set $x$ in a generic extension will be denoted by $\name{x}$, and a canonical name
for an element $x$ in the ground model $V$ will be denoted by $\can{x}$. 
In certain cases we will write $V^{\po}$ to denote a generic extension of $V$ by a generic filter of $\po$.

\section{Tame Orders}\label{Section - I - Tame Orderings}
We define a family of orders called Tame Orders which is closely related to the orders $R_\lambda$, $\lambda \in \On$, introduced above.
The relation between tame orders and the orders $R_\lambda$ is given in 
Proposition \ref{Proposition - tame} and makes use of the notions of
\emph{reduced orders} (Definition \ref{Definition - DU sets}) and \emph{tame ranks} (Definition \ref{Definition - Tame Rank}).

\begin{definition}${}$
\begin{enumerate}
\item $(R_{2,2},<_{R_{2,2}})$ is an order on a set of four elements 
$R_{2,2} = \{x_0,x_1,y_0,y_1\}$, defined by 
$<_{R_{2,2}} = \{ ( x_0,y_0) , ( x_1,y_1 )\}$.

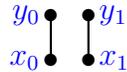
\newcommand*{\zc}{0}%
 \newcommand*{\oc}{1}%
\begin{center}
\begin{tikzpicture}[xscale=0.5, yscale=0.6]
    \draw  node [] at (\zc,\zc) {$\bullet$};
      \draw  node [left] at (\zc,\zc) {\blue{$x_0$}};
      \draw  node [] at (\zc,\oc) {$\bullet$};
      \draw  node [left] at (\zc,\oc) {\blue{$y_0$}};
      \draw  node [] at (\oc,\zc) {$\bullet$};
      \draw  node [right] at (\oc,\zc) {\blue{$x_1$}};
      \draw  node [] at (\oc,\oc) {$\bullet$};
      \draw  node [right] at (\oc,\oc) {\blue{$y_1$}};
      \draw [thick,black] (\zc,\zc) -- (\zc,\oc);      
      \draw [thick,black] (\oc,\zc) -- (\oc,\oc);            
\end{tikzpicture}
\end{center}

\item $(S_{\omega,2},<_{S_{\omega,2}})$ is an order on a disjoint union of two  countable sets
$S_{\omega,2} = \{x_n\}_{n < \omega} \uplus \{y_n \}_{n<\omega}$, defined by
$<_{S_{\omega,2}} = \{ (x_m, y_n) \mid m \geq n\}$.

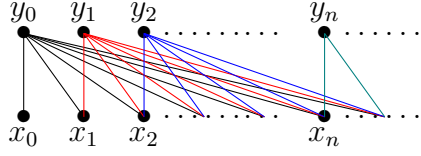
\begin{center}
\begin{tikzpicture}[xscale=0.8, yscale=0.37]

    \pgfmathtruncatemacro\mult{1} 
    
    \draw [step=1.0mm,thin,black]  node [] at (0,0) {$\bullet$};
    \draw [step=1.0mm,thin,black]  node [below] at (0,0) {$x_0$};
    
    \draw [step=1.0mm,thin,black]  node [] at (1,0) {$\bullet$};
    \draw [step=1.0mm,thin,black]  node [below] at (1,0) {$x_1$};
    
    \draw [step=1.0mm,thin,black]  node [] at (2,0) {$\bullet$};
    \draw [step=1.0mm,thin,black]  node [below] at (2,0) {$x_2$};
    
    \draw [step=1.0mm,thin,black]  node [] at (3,0) {$\dots\dots$};
    \draw [step=1.0mm,thin,black]  node [] at (4,0) {$\dots$};
    
    \draw [step=1.0mm,thin,black]  node [] at (5,0) {$\bullet$};
    \draw [step=1.0mm,thin,black]  node [below] at (5,0) {$x_n$};
    
        \draw [step=1.0mm,thin,black]  node [] at (6,0) {$\dots\dots$};
    
        \draw [step=1.0mm,thin,black]  node [] at (0,3) {$\bullet$};
    \draw [step=1.0mm,thin,black]  node [above] at (0,3) {$y_0$};
    
    \draw [step=1.0mm,thin,black]  node [] at (1,3) {$\bullet$};
    \draw [step=1.0mm,thin,black]  node [above] at (1,3) {$y_1$};
    
    \draw [step=1.0mm,thin,black]  node [] at (2,3) {$\bullet$};
    \draw [step=1.0mm,thin,black]  node [above] at (2,3) {$y_2$};
    
    \draw [step=1.0mm,thin,black]  node [] at (3,3) {$\dots\dots$};
    \draw [step=1.0mm,thin,black]  node [] at (4,3) {$\dots$};
    
    \draw [step=1.0mm,thin,black]  node [] at (5,3) {$\bullet$};
    \draw [step=1.0mm,thin,black]  node [above] at (5,3) {$y_n$};
    
   \draw [step=1.0mm,thin,black]  node [] at (6,3) {$\dots\dots$};

     \foreach \i in {0,...,6} {
        \draw [thin,black] (0,3) -- (\i,0) ;}

     \foreach \i in {1,...,6} {
        \draw [thin,red] (1,3) -- (\i,0) ;}
        
     \foreach \i in {2,...,6} {
        \draw [thin,blue] (2,3) -- (\i,0) ;}
        
     \foreach \i in {5,...,6} {
        \draw [thin,teal] (5,3) -- (\i,0) ;}

\end{tikzpicture}
\end{center}

\end{enumerate}
\end{definition}

\noindent Let $(S,<_S)$ and $(R,<_R)$ be two orders.
An injection $\pi : S \to R$ is an \textit{embedding} of $(S,<_S)$ into $(R,<_R)$ if it compatible and incompatible preserving. We say that $(R,<_R)$ embeds $(S,<_S)$.

\begin{definition}[Tame Orders]\label{Definition - Tame Orders}${}$\\
An order $(S,<_S)$ is \emph{tame} if it does not embed $R_{2,2}$ nor $S_{\omega,2}$.
\end{definition}

The rest of this section is devoted to describing how tame orders relates to the orders $R_\lambda$, $\lambda \in \On$.

\begin{definition}\label{Definition - DU sets}
Let $(S,<_S)$ be an order.
\begin{enumerate}
\item 
For $x \in S$ let $d(x) = \{ z \in S \mid z <_S x\}$ and $u(x) = \{ z \in S \mid x <_S z\}$.
\item Let $\sim_S$ be the equivalence relation on $S$, define by $x \sim_S y$ if and only if
$(d(x),u(x)) = (d(y),u(y))$
\item 
$(S,<_S)$ is \textit{reduced} if and only if there are no distinct $\sim_S$ equivalent elements in $S$.

\item For any $(S,<_S)$ let $([S],<_{[S]}) = (S,<_S)/ \sim_S$ be the induced order on $\sim_S$ equivalent classes. $([S],<_{[S]})$ is clearly reduced. 
\end{enumerate}
\end{definition}

The main result in this section (Proposition \ref{Proposition - tame}) shows that a well-founded reduced order $(S,<_S)$ is tame if and only if it embeds in $R_\lambda$ for some $\lambda < |S|^+$.

\begin{lemma}\label{Observation - key tame observation}
Let $(S,<_S)$ be an order. The following are equivalent:
\begin{enumerate}
\item\label{1} $(S,<_S)$ does not embed $R_{2,2}$.
\item\label{2} For every $x,x' \in S$, the sets $u(x)$,$u(x')$ are $\subseteq-$comparable.
\item\label{3} For every $x,x' \in S$, the sets $d(x)$,$d(x')$ are $\subseteq-$comparable.
\end{enumerate}
\end{lemma}

\begin{proof}
In $R_{2,2}$, the sets $u(x_0),u(x_1)$ are $\subseteq-$incomparable as $y_0 \in u(x_0) \setminus u(x_1)$ and 
$y_1 \in u(x_1)\setminus u(x_0)$. If $\pi : R_{2,2} \to S$ is an embedding it follows that $u(\pi(x_0)), u(\pi(x_1))$ are $\subseteq-$incomparable. Therefore \ref{2} implies \ref{1}. The fact that
$d(y_0)$ and $d(y_1)$ are $\subseteq-$incomparable is similarly used to show that \ref{3} implies \ref{1}.\\
Suppose now that $(S,<_S)$ does not embed $R_{2,2}$, and 
let $x,x' \in S$. {If $u(x)$ and $u(x')$ were $\subseteq-$incomparable,
 there would be some $y,y'$ so that }
\begin{enumerate}
 \item $x <_S y$, $x' \nless_S y$, and
 \item $x' <_S y'$, $x \nless_S y'$. 
\end{enumerate}
This is impossible as it would imply that
$<_S \uhr \{x,x',y,y'\}$ is isomorphic to $R_{2,2}$. It follows that
\ref{1} implies \ref{2}.
The proof of that \ref{1} implies \ref{3} is similar.
\end{proof}

\begin{definition}\label{Definition - I - CU stuff}
Let $(S,<_S)$ be an order.
\begin{enumerate}
\item For every $x \in S$ define $\cu(x) = S \setminus u(x) = 
\{y \in S \mid x \nless_S y\}$.

\item Define $D(S) = \{ d(x) \mid x \in S\}$ and 
$\CU(S) = \{ \cu(x) \mid x \in S\}$. 

\end{enumerate}
\end{definition}

\noindent For $x,y \in S$, $u(x) \rsb u(y)$ if and only if $\cu(x) \subseteq \cu(y)$, hence $(\CU(S),\subseteq) \cong (U(S),\rsb)$, where
$U(S) = \{u(x) \mid x \in S\}$.
By Lemma \ref{Observation - key tame observation},
if $(S,<_S)$ does not embed $R_{2,2}$ then 
$(D(S),\subseteq)$, $(\CU(S),\subseteq)$ are linear. 

\begin{lemma}
The following are equivalent for a well founded order $(S,<_S)$ which does not embed $R_{2,2}$:
\begin{enumerate}\label{Lemma - almost tame}
\item\label{11}  $(S,<_S)$ does not embed $(S_{\omega,2},<_{S_{\omega,2}})$.
\item\label{22} $(D(S),\subsetneq)$ is well founded.
\item\label{33} $(\CU(S),\subsetneq)$ is well founded.
\end{enumerate}
\end{lemma}

\begin{proof}
In $S_{\omega,2}$, the sequence $\la d(y_n) \mid n < \omega\ra$ is a $\subseteq-$strictly decreasing as witnessed by $\la x_n \mid n < \omega\ra$. Suppose that $\pi : S_{\omega,2} \to S$ is an embedding of $(S_{\omega,2}, <_{S_{\omega,2}})$ into $(S,<_S)$. Then $\pi(x_n) \in d(\pi(y_n))\setminus d(\pi(y_m))$ for every $n < m < \omega$. But $D(\pi(y_n)),D(\pi(y_m))$ are $\subseteq-$comparable as
$(S,<_S)$ does not embed $R_{2,2}$, so it must be that $d(\pi(y_m)) \subsetneq d(\pi(y_n))$. Therefore $\la d(\pi(y_n)) \mid n < \omega\ra$ is $\subseteq-$strictly decreasing.  This shows \ref{22} implies \ref{11}. 
Similarly the fact that in $S_{\omega,2}$, $\la \cu(x_n) \mid n < \omega\ra$ is a $\subseteq$ strictly decreasing sequence, is used to show that \ref{33} implies \ref{11}.\\
Suppose now that $(S,<_S)$ is a well founded order which does not embed $R_{2,2}$ {and fails to satisfy} \ref{22}. Let $\la y_n \mid n <\omega\ra$ be a sequence of distinct elements in $S$ such that $\la d(y_n) \mid n < \omega\ra$ is $\subseteq-$strictly decreasing. Since $S$ is well founded we may assume that $y_m \nless_S y_n$ for every $n < m$. {Furthermore, we cannot have  $y_n <_S y_m$ as it would imply that} $y_n \in d(y_n)$. It follows that the elements in $\la y_n \mid n < \omega\ra$ are $<_S$ pairwise incomparable. \\
Next, for each $n <\omega$ pick $x_n \in d(y_n) \setminus d(y_{n+1})$, thus
$x_n \in d(y_n) \setminus d(y_m)$ for every $n < m$.
We can thin out the sequence $\la x_n \mid n < \omega \ra$ to get an infinite subsequence so that $x_m \nless_S x_n$ {whenever $m > n$ and $x_n,x_m$ are members of the subsequence}.
For simplicity let us assume that $x_m \nless_S x_n$ for every $n < m$. We claim that also $x_n \nless_S x_m$ for every $n < m$. For this, note that $d(y_m),d(x_m)$ are $\subseteq$ compatible and $x_m \in d(y_m) \setminus d(x_m)$, so $d(x_m) \subseteq d(y_m)$. Therefore if $x_n <_S x_m$, $x_n \in d(x_m) \subseteq d(y_m)$ which contradicts our choice of $x_n$. We conclude that the elements in the sequence $\la x_n \mid n < \omega\ra$ are also $<_S$ pairwise incomparable. \\
We claim that $<_S\uhr (\{x_n\}_{n<\omega} \uplus \{y_n\}_{n<\omega})$ is isomorphic to $S_{\omega,2}$. 
It remains to show that the sets $\{x_n\}_{n<\omega}$, $\{y_n\}_{n<\omega}$ are disjoint.
To this end, we have that $x_n \neq y_n$ for all $n < \omega$, as $x_n \in d(y_n)$. Also if $m \neq n$ then $x_n \neq y_m$ as otherwise $y_n,y_m$ would be $<_S$ comparable in contradiction to the above. 
It follows that \ref{11} implies \ref{22}. {Using a similar argument one can show that \ref{11} implies \ref{33}. } 
\end{proof}

\begin{definition}\label{Definition - I - bar D}
Let $(S,<_S)$ be an order.
\begin{enumerate}
\item 
Define $\D(S)$ be the completion of $D(S)$ under $\subseteq$ increasing sequences namely $d \in \D(S) \setminus D(S)$ if and only if $d = \cup C$ for some $C \subset D(S)$ which is $\subsetneq-$downward closed, i.e. 
for all $d_1,d_2 \in D(S)$, if $d_2 \in C$ and $d_1 \subseteq d_2$ then $d_1 \in C$.
  
 \item For every $x \in S$ let 
 \begin{itemize}
 \item $<_{\D}(x) = \{d \in \D(S) \mid d \subsetneq d(x)\}$, and 
 \item $<_{\CU}(x) = \{\cu(y) \in \CU(S) \mid \cu(y) \subsetneq \cu(x)\}$.
 \end{itemize}
\end{enumerate}
\end{definition}

\begin{remarks}\label{Remarks - Tame}${}$
\begin{enumerate}
\item Note that $D(S) \subset \D(S)$. Indeed for every $x \in S$, 
$d(x) = \cup C_x$, where $C_x = \{d \in D(S) \mid d\subseteq d(x)\}$ is
$\subseteq-$downward closed. 

\item The elements $d \in \D(S) \setminus D(S)$ are $\subseteq-$limits of $D(S)$. Therefore if $d \in \D(S)$ has a $\subseteq$ immediate successor $d^+ \in \D(S)$ then $d^+$ is not a $\subseteq-$limit and therefore $d^+ \in D(S)$ is of the form $d^+ = d(z)$ for some $z \in S$.

\item  
If $(S,<_S)$ is a well order which does not embed $R_{2,2}$ nor
$S_{\omega,2}$ then $(D(S),\subseteq)$ and $(\CU(S),\subseteq)$ are well orders. 
{Since $\D(S)$ introduces only $\subseteq-$limits to $D(S)$, it follows that   
$(\D(S),\subseteq)$ is a well order}. 

\item 
For every $x \in S$, the sets $<_{\D}(x)$, $<_{\CU}(x)$ are $\subseteq$ initial segments of $\D(S)$, $\CU(S)$ respectively. In particular
$(<_{\D}(x), \subseteq)$ and $(<_{\CU}(x), \subseteq)$ are well orders.
\end{enumerate}
\end{remarks}

\begin{definition}\label{Definition - Tame Rank}
Let $(S,<_S)$ be a tame order. We define the \emph{tame rank} of $(S,<_S)$ to be
the ordertype of the well-ordered set $(\CU(S),\subset)$, and denote it by 
$\tamerank(S,<_S)$
\end{definition}

It is not difficult to see that for every tame order $(S,<_S)$,
\[ \rank(S,<_S) \leq \tamerank(S,<_S) < |S|^+\]

\begin{proposition}\label{Proposition - tame}
A well-founded reduced order $(S,<_S)$ is tame if and only if it embeds in some $R_\lambda$.
Moreover, $\lambda = {\tamerank(S,<_S)} < |S|^+$ is the minimal embedding ordinal.
\end{proposition}

\begin{proof}(\textbf{Proposition \ref{Proposition - tame}})\\
Let $\lambda \in \On$ and $X \subseteq R_\lambda$. The order
$(X,<_{R_\lambda}\uhr X)$ is clearly well-founded and reduced. To show that it does not embed $R_{2,2}$, $S_{\omega,2}$, it is sufficient to check that $R_\lambda$ {does not embed these orders}:

\begin{enumerate}
\item ($R_\lambda$ does not embed $R_{2,2}$)\\
Let $(a_0,b_0),(A_0,B_0),(a_1,b_1),(A_1,B_1)$ be four elements in $R_\lambda$ and suppose that $(a_i,b_i) <_{R_\lambda} (A_i,B_i)$ for $i \in \{0,1\}$ (i.e., they satisfy all $R_{2,2}$ relations). We claim that $<_{R_\lambda}$ must satisfy an additional relation which is not compatible with $R_{2,2}$. Indeed if $i \in \{0,1\}$ satisfies $A_i = \max(A_0,A_1)$ 
then $b_0,b_1 < A_i$. Hence both $(a_0,b_0),(a_1,b_1)$ are $<_{R_\lambda}$ smaller than $(A_i,B_i)$.

\item ($R_\lambda$ does not embed $S_{\omega,2}$)\\
Let $X = \{ (m_n^x,M_n^x) \mid n < \omega\} \cup  \{ (m_n^y,M_n^y) \mid n < \omega\} \subset R_{\lambda}$. Let $\pi : S_{\omega_2} \to R_\lambda$ defined by $\pi(x_n) = (m_n^x,M_n^x)$ and $\pi(y_n) = (m_n^y,M_n^y)$, $n < \omega$. We claim that $\pi$ cannot be an embedding of $(S_{\omega,2},<_{S_{\omega,2}})$ into $(R_{\lambda},<_{R_{\lambda}})$. Otherwise 
setting $m^* = \min(\{m_n^y \mid n < \omega\}) <\lambda$ and $n^* = \min(\{n < \omega \mid m_n^y = m^*\}$, we get that
$\pi(x_{n^*}) <_{R_{\lambda}} \pi(y_{n^*})$,
i.e., $M^x_{n^*} < m^y_{n^*} = m^*$. It follows that for every
$n > n^*$, $M^x_{n^*} < m^* \leq m^y_n$ thus $\pi(x_{n^*}) <_{R_\lambda} \pi(y_n)$. But this is incompatible with $<_{S_{\omega,2}}$. 
\end{enumerate}
\noindent It follows that $(R_\lambda,<_{R_\lambda})$ is tame, and it is easy to see that $\tamerank(R_\lambda) = \lambda$. Therefore if $(S,<_S)$ is tame and embeds in $R_\lambda$ then $\tamerank(S,<_S) \leq \lambda$. \\

\noindent Next, {suppose that $(S,<_S)$ is a reduced well-founded order which does not embed
$R_{2,2}$ nor $S_{\omega,2}$}, and let $\lambda = \tamerank(S,<_S)$.
Define functions, $m,M : S \to \lambda$ by
\[ m(x) = \otp(<_{\D}(x), \subseteq) \quad \text{ and } \quad M(x) = \otp(<_{\CU}(x), \subseteq).\]
We claim that the map $\pi : S \to R_\lambda$, defined by 
$\pi(x) = (m(x),M(x))$, is an embedding of $(S,<_S)$ into $(R_\lambda,<_{R_\lambda})$. \\
$(S,<_S)$ is reduced, therefore for every distinct $x,y \in S$, $(d(x),\cu(x)) \neq (d(y),\cu(y))$. We get that
one of  $<_{\D}(x)$, $<_{\D}(y)$ is a $\subseteq-$strict initial segment of the other, or, one of  $<_{\CU}(x)$, $<_{\CU}(y)$ is a $\subseteq-$strict initial segment of the other. Hence $m(x) = \otp(<_{\D}(x), \subseteq) \neq \otp(<_{\D}(y), \subseteq) = m(y)$, or,
$M(x) = \otp(<_{\CU}(x), \subseteq) \neq \otp(<_{\CU}(y), \subseteq) = M(y)$.
Therefore $\pi$ is injective.\\

\noindent The next three claims show that $\pi$ is order preserving:\\

\noindent\textbf{Claim 1 -}
For every $x \in S$ there exists $\subseteq-$order preserving injection 
$f: <_{\D}(x) \to <_{\CU}(x)$, thus $m(x) \leq M(x)$. 

Let $d \in <_{\D}(x)$ and define $f(d)$ as follows:
Let $d^+ \in \D(S)$ be the $\subseteq$ immediate successor of $d$, i.e., 
$d \subsetneq d^+ \subseteq d(x)$. Pick an element $y_d \in d^+ \setminus d$ 
and set $f(d) = \cu(y_d)$.\\
To show that $f(d)$ belongs to $<_{\CU}(x)$, note that $y_d \in d^+ \subset d(x)$, so 
$y_d <_S x$ and $x \in \cu(x)\setminus \cu(y_d)$. As $\cu(x)$ and $\cu(y_d)$ are $\subseteq$ comparable {it must mean that} $\cu(y_d) \subsetneq \cu(x)$.\\
As pointed out in Remarks \ref{Remarks - Tame}, we have that $d^+ = d(z)$ for some $z \in S$. Therefore $y_d <_S z$. 
Let $d' \in <_{\D}(x)$ so that $d \subsetneq d'$. $d^+ = d(z) \subset d' \subsetneq (d')^+$ implies that $y_{d'} \nless_S z$ and therefore $z \in \cu(y_{d'}) \setminus \cu(y_d)$, so $f(d) = \cu(y_d) \subsetneq \cu(y_{d'}) = f(d')$. Hence $f$ is $\subsetneq-$order preserving.\\

\noindent\textbf{Claim 2 -}
For every $x <_S y$ there is a $\subseteq-$order preserving injection 
$g: <_{\CU}(x) \to <_{\D}(y)$, witnessing that $M(x) < m(y)$.

Let $\cu(z) \in <_{\CU}(x)$ and define $g(\cu(z))$ as follows:
Let $z^+ \in S$ so that $\cu(z^+)$ is the $\subseteq$ immediate successor
of $\cu(z)$. 
Let $\Gamma_{\cu(z)} = \{d(w) \mid w \in \cu(z^+) \setminus \cu(z)\}$ and let $g(\cu(z)) \in \Gamma_{\cu(z)}$  be a $\subseteq$ minimal set in $\Gamma_z$ (it is actually unique). Also pick $w_{\cu(z)} \in \cu(z^+) \setminus \cu(z)$ so that 
$g(\cu(z)) = d(w_{\cu(z)})$. We get that $w_{\cu(z)} \nless_S x$ as $w_{\cu(z)} \in \cu(z^+) \subseteq \cu(x)$, and $z <_S w_{\cu(z)}$ as $w \not\in \cu(z)$. \\
For every $w' \in S$, if $w' <_S w_{\cu(z)}$ then $w' <_S y$ as otherwise $<_S\uhr \{w_{\cu(z)},w',x,y\}$ would be isomorphic to $R_{2,2}$. It follows that
$g(\cu(z)) = d(w_{\cu(z)}) \subseteq d(y)$. Moreover $d(w_{\cu(z)}) \subsetneq d(y)$ since $x \in d(y)\setminus d(w_{\cu(z)})$. This shows that $g(\cu(z)) \in <_{\D}(y)$.\\
Let $z' \in S$ so that $\cu(z) \subsetneq \cu(z') \in <_{\CU}(x)$.
We have $w_{\cu(z)} \in \cu(z^+) \subseteq \cu(z')$, i.e. $z' \nless_S w_{\cu(z)}$ and therefore $z' \in d(w_{\cu(z')}) \setminus d(w_{\cu(z)})$.
It follows that $g(\cu(z)) = d(w_{\cu(z)}) \subsetneq d(w_{\cu(z')}) = g(\cu(z'))$. Therefore $g$ is $\subseteq-$order preserving.\\
Suppose that $\otp(<_{\D}(y),\subseteq) = \rho + n$ where $\rho$ is a limit ordinal
and $n < \omega$. Let $\la d_i \mid i < \rho + n\ra$ be a $\subseteq-$continuous increasing enumeration of $<_{\D}(y)$. {In order  to prove $M(x) < m(y)$ it is sufficient to verify that $d_\rho = \bigcup_{i<\rho}d_i \not\in \rng(g)$}.
We consider the following three cases which address the identity of $d_\rho$:

\begin{enumerate}
 \item If $\rho = 0$ then $d_\rho = \emptyset$, as $\emptyset = d(t) \in <_{D}(x)$ for every $<_S$ minimal element $t <_S x$. We saw that $x \in d(w_{\cu(z)}) = g(\cu(z))$ for every $\cu(z) \in <_{\CU}(x)$, therefore $g(\cu(z)) \neq\emptyset$ for every $\cu(z) \in \dom(g)$.
 
 \item If $d_\rho \in \D(S)\setminus D(S)$ then $d_\rho \not\in \rng(g)$ since $g(\cu(z)) = d(w_{\cu(z)}) \in D(S)$ for every $\cu(z) \in <_{\CU}(x)$.
 
 \item Suppose that $d_\rho = d(w)$ for some $w \in S$.
 Recall that for every $\cu(z) \in <_{\CU}(x)$, $g(\cu(z))$ is  a $\subseteq$ minimal set in $\Gamma_{\cu(z)} = \{d(w) \mid w \in \cu(z^+) \setminus \cu(z)\}$. Therefore to show $d_\rho \not\in \rng(g)$ it is sufficient to verify $d_\rho$ is not $\subseteq-$minimal in $\Gamma_{\cu(z)}$ for any
 $\cu(z) \in <_{\CU}(x)$.
For this, note that $d_\rho = d(w) \in \Gamma_{\cu(z)}$ implies $z \in d_\rho = \bigcup_{i < \rho}d_i$. Let $i < \rho$ be a successor ordinal such that $z \in d_i = d(w_i)$. Since $z^+ \not\in d_\rho$ and $d(w_i) \subseteq d_\rho$, we get that $w_i  \in \cu(z^+) \setminus \cu(z)$,  witnesses that $d_\rho$ is not $\subseteq-$minimal.
  \end{enumerate}

\noindent \textbf{Claim 3 -}
For every $x \nless_S y$ there is a $\subseteq-$preserving injection 
$h:  <_{\D}(y) \to <_{\CU}(x)$, thus $m(y) \leq M(x)$.

Let $d \in <_{\D}(y)$ and define $h(d)$ as follows:
Let $d^+ \in \D(S)$ be the $\subseteq$ immediate successor of $d$, i.e., 
$d \subsetneq d^+ \subseteq d(y)$. Pick an element $w_d \in d^+ \setminus d$ 
and set $h(d) = \cu(w_d)$. Since $\cu(w_D)$, $\cu(y)$ are $\subseteq$ comparable, and $y \in \cu(x)\setminus \cu(w_d)$ (as $x \nless_S y$) we get 
$h(d) = \cu(w_d) \subsetneq \cu(x)$, i.e., $h(d) \in <_{\cu}(x)$. \\
Let $d' \in <_{\D}(y)$ so that $d \subsetneq d'$. $d^+ = d(z)$ for some $z \in S$. 
We have $d(z) \subset d'$ and $w_{d'} \not\in d'$, so $
z \in \cu(w_{d'})\setminus \cu(w_d)$ and thus $h(d) = \cu(w_d) \subsetneq \cu(w_{d'}) = h(d')$. Therefore $h$ is $\subsetneq-$order preserving.
\end{proof}

\begin{remark}${}$
\begin{enumerate}
\item {Definition \ref{Definition - Tame Orders} (of tame orders) is slightly different from the author's original (equivalent) definition, where $S$ is tame, if it does not embed $R_{2,2}$ and the linear ordered sets $(\D(S),\subset)$, $(\CU(S),\subset)$ are well-orders. The author would like to thank the referee for pointing out that the last property is equivalent to the fact that $S$ does not embed $S_{\omega,2}$ as well.}

\item In \cite{Cummings I}, James Cummings constructed a model in which $\mo(\kappa)$ is divided into blocks $\{M(\alpha,\beta) \mid \alpha < o(\kappa) ,\beta \in (\alpha,o(\kappa)) \cup \{\infty\}\}$.
The blocks determine the $\mo$ structure in this model, where for every $U'\in M(\alpha',\beta')$ and $U \in M(\alpha,\beta)$, $U' \mo U$ if and only if 
$\beta' \leq \alpha$. It is not difficult to see that this order is tame.

\end{enumerate}
\end{remark}

\section{The Posets $\pzero$ and $\pone$}\label{Section - I - Posets}

The purpose of this section is to introduce the main poset $\po = \po^0 * \po^1$, comprised of
$\po^0$- a Friedman Magidor forcing, introduced in  \cite{Friedman-Magidor - Normal Measures}, and of
 $\po^1$ - a Magidor iteration of Prikry type forcings. The Magidor iteration of Prikry type forcings
 was introduced by Magidor in \cite{Magidor} (See \cite{Gitik-HB} for an extensive survey). 
 The definitions of $\po^0$ and $\po^1$ rely on certain parameters, chosen relative to $\lambda = o^V(\kappa)$.
 To simplify the presentation we restrict the presentation of $\pzero$,$\pone$ in this section to when $\lambda \leq\kappa$.  The more general case, $\lambda < \kappa^+$, will be treated in section \ref{Section - I - final cut}.\\
 Suppose that the ground model is a  Mitchell model, $V = L[\U]$, so that
 $\U = \la U_{\nu,\tau} \mid \nu \leq \kappa, \tau < o(\alpha)\ra$ is a coherent sequence of normal measures.
 When $\nu = \kappa$ we write $U_{\tau}$ to denote $U_{\kappa,\tau}$, for every $\tau < \lambda$. 
 For every $\tau < \lambda$ let $\Delta_\tau = \{\nu < \kappa \mid o(\nu)= \tau\}$.
 The sets $\{\Delta_\tau \mid \tau < \lambda\}$ are pairwise disjoint as $\lambda \leq \kappa$.
 For every $\alpha < \lambda$, let $j_\alpha : V \to M_\alpha \cong \Ult(V,U_\alpha)$ be the induced
 ultrapower embedding. Therefore $U_\alpha \in M_\beta$ if and only if $\alpha < \beta < \lambda$.
 Let  $j_\alpha^{M_\beta} : M_\beta \to M_{\alpha,\beta} \cong \Ult(M_\beta,U_\alpha)$, and
 $j_{\alpha,\beta} = j_\alpha^{M_\beta} \circ j_\beta : V \to M_{\alpha,\beta}$.
 {$j_{\alpha,\beta}$ is known to be equivalent to the ultrapower embedding induced by the product measure $U_\alpha\times U_\beta$.}
 $j_{\alpha,\beta}$ can also be formed using a normal iteration
 Taking $j_\alpha(U_\beta) \in M_\alpha$, if $i_{\beta}^{M_\alpha} : M_\alpha \to M_{\alpha,\beta} \cong \Ult(M_\alpha,j_\alpha(U_\beta))$ 
 denote the induced ultrapower embedding of $M_\alpha$ by $j_\alpha(U_\beta)$ then
 $j_{\alpha,\beta} = i_\beta^{M_\alpha} \circ j_\alpha$. 
 The iteration is normal since the sequence of critical points $\la \kappa_0, \kappa_1 \ra = \la \kappa, j_\alpha(\kappa)\ra$ is increasing. 
The fact that  $j_\beta(\kappa) > \kappa$ is inaccessible in $M_\beta$ implies that
 $j_{\alpha,\beta}(\kappa) = j_\alpha^{M_\beta}(j_\beta(\kappa)) = j_{\beta}(\kappa)$.
 
\subsection{The Poset $\pzero$}\label{SubSection - I - pzero} 
The forcing $\pzero$ used here was introduced by Friedman and Magidor in \cite{Friedman-Magidor - Normal Measures}.
$\pzero = \pzero_{\kappa+1} = \la \pzero_\nu, \qzero_\nu \mid \nu \leq \kappa\ra$
is a non-stationary support iteration.
Conditions $p \in \pzero_\nu$ are denoted by $p = \la \name{p_\mu} \mid \mu < \nu \ra$.
Non-stationary support means that for every limit $\nu \leq \kappa$, every $p \in \pzero_\nu$ belongs to the inverse limit of the posets $\la \pzero_\mu \mid \mu < \nu\ra$,
with the restriction that if $\nu$ is inaccessible then the set of $\mu<\nu$ such that $p_\mu$
is nontrivial is a non stationary subset of $\nu$.
For every $\nu \leq \kappa$, if $\nu$ is non-inaccessible then $\force_{\pzero_\nu} \qzero_\nu = \emptyset$,
otherwise $\force_{\pzero_\nu} \qzero_\nu = \Sacks_{\lambda(\nu)}(\nu)*\Code(\nu)$,
where for every $\nu < \kappa$ we set 
\[\lambda(\nu) = 
\begin{cases}
  \lambda &\mbox{if }  \lambda < \kappa,  \\
  \nu &\mbox{if } \lambda = \kappa.
\end{cases}\]
Conditions in $\Sacks_{\nu}(\nu)$ are trees $T \subset {}^{<\nu}\nu$ for which there is a
closed unbounded $C \subset \nu$ so that whenever $s \in T$, if $\len(s) \in C$
then $s \fr \la i \ra \in T$ for all $i < \lambda(\len(s))$. 
$\Code(\nu)$ a coding posets which adds a club set to $\nu^+$ coding the $\Sacks_{\lambda(\nu)}(\nu)$ generic
function $s_\nu : \nu \to \nu$ by destroying stationary subsets of $\Cf(\nu) \cap \nu^+$ in $V$ 
{(see \cite{Friedman-Magidor - Normal Measures}.}
Let $\Gzero \subset \pzero$ be a generic filter over $V$, and for every $\nu \leq \kappa$ let
$\Gzero(\qzero_\nu) = \bigcup\{ \name{p_\nu}_{G^0\uhr \nu} \mid p \in \Gzero\}$ be the induced $\qzero_\nu-$generic over $V[\Gzero\uhr\nu]$. 
For every non trivial stage $\nu \leq \kappa$ in the iteration $\pzero$, let $s_\nu : \nu \to \lambda(\nu)$ denote
the generic Sacks function specified by $\Gzero(\qzero_\nu)$. 
For every $\eta < \lambda$ define 
\[ \delset(\eta) = \{\nu < \kappa \mid s_\nu = s_\kappa\uhr\nu, \text{ and } \thinspace s_\kappa(\nu) = \eta\}. \]
The sets in $\{\Delta(\eta) \mid \eta < \lambda\}$ are clearly pairwise disjoint.
According to Friedman and Magidor \cite{Friedman-Magidor - Normal Measures} $V[G^0]$ satisfies the following properties:
\begin{fact}\label{Fact - I - Friedman Magidor}
\begin{enumerate}
\item $V[G^0]$ agrees with $V$ on all cardinals and cofinalities.
\item {For every normal measure $U$ on $\kappa$ and $\eta < \lambda$, 
there is a unique normal measure, $U(\eta) \in V[G^0]$ containing $U \cup \{\Delta(\eta)\}$. Furthermore, 
these are the only normal measures on $\kappa$ in $V[G^0]$.}

\item For every $\eta < \lambda$, let 
$j_{U(\eta)} : V[G^0] \to M_{U(\eta)}^0 \cong \Ult(V[G^0],U(\eta))$ be the induced ultrapower
embedding of $V[G^0]$ by $U(\eta)$. We have
\begin{enumerate}
 \item $j_{U(\eta)}\uhr V = j_U : V \to M_U \cong \Ult(V,U)$ is the induced ultrapower
 embedding of $V$ by $U$.
 \item $M_{U(\eta)}^0 = M_U[G^0_{U(\eta)}]$ where $G^0_{U(\eta)} \subset j_U(\po^0)$ is
 $M_U$ generic.
 \item $G^0_{U(\eta)}\uhr \kappa+1 = \Gzero$.
 \item $\bigcup j_U`` (G^0\uhr\kappa)$ completely determines $G^0_{U(\eta)}\uhr j_U(\kappa) \subset j_U(\po^0)\uhr j_U(\kappa)$.
 In particular $G^0_{U(\eta)}\uhr j_U(\kappa)$ is independent of $\eta < \lambda$. 
 \item For every $p \in \Gzero$ let $j_U(p)(\eta)$ be the condition obtained
 from $j_U(p)$ by reducing its Sacks tree at level $\kappa$, $T = (j_U(p))_\kappa$, to the set of functions 
 $s \in Lev_{\kappa+1}(T)$ which satisfy $s(\kappa) = \eta$.
 For every dense open set $D \subset j_U(\pzero)$, there exists  $p \in \Gzero$ so that
 $j_U(p)(\eta) \in D$.
\end{enumerate}
\end{enumerate}
\end{fact}

\begin{definition}\label{Definition - Uz alpha beta and jz alpha beta}
\begin{enumerate}
 \item Let $\Uz{\alpha}{\eta}$ denote the $V[\Gzero]$ extension $U(\eta)$ described above, for $U = U_\alpha$.
 \item Let $\jz{\alpha}{\eta} : V[\Gzero] \to M_\alpha[\Gzero_\alpha(\eta)]$ denote the ultrapower
 embedding of $V[\Gzero]$ by $\Uz{\alpha}{\eta}$. 
\end{enumerate}
\end{definition}
 We write  $\Gzero_\alpha(\eta) = \jz{\alpha}{\eta}(\Gzero)$ to 
 denote the $j_\alpha(\pzero)$ generic filter $\Gzero_{U(\eta)}$ described above with $U = U_\alpha$. 
 It is clear that $\jz{\alpha}{\eta}\uhr V = j_\alpha$. 
 
 \begin{definition}\label{Definition - Delta_alpha sets}${}$
 \begin{enumerate}
 \item  For every $\alpha < o(\kappa)$ let $\Delta_\alpha = \{\nu < \kappa \mid o(\nu) = \alpha\}$.
 \item For $\alpha < o(\kappa)$ and $\eta < \lambda$ let  let $\delset_\alpha(\eta) = \delset(\eta) \cap \delset_\alpha$.
 \end{enumerate}
 \end{definition}

\noindent Based on the above definition, it immediately follows that:
\begin{enumerate}
 \item $\{\delset_\alpha(\eta) \mid \alpha,\eta < \lambda\}$ are pairwise disjoint. 
 \item $\delset_{\alpha}(\eta) \in \Uz{\alpha}{\eta}$ for every $\alpha,\eta < \lambda$. 
\end{enumerate}

\noindent 
We can easily describe the iterated ultrapowers using $\Uz{\alpha}{\eta}$, $\alpha,\eta < \lambda$.
According to \cite{Friedman-Magidor - Normal Measures},
$j_{\alpha,\beta}``\Gzero$ determines a unique generic filter in $j_{\alpha,\beta}(\pzero)$ up
to the following values:
\begin{enumerate}
\item  A tuning fork at the value $s_{j_\alpha(\kappa)}(\kappa) < \lambda$, and
\item A tuning fork at the value $s_{j_{\alpha,\beta}(\kappa)}(j_\alpha(\kappa)) < j_\alpha(\lambda)$. 
\end{enumerate}
For every $\eta_\alpha < \lambda$ and $\eta_\beta < j_\alpha(\lambda)$ there exists a unique extension of $j_{\alpha,\beta}$ to an embedding of $V[\Gzero]$, which determines $s_{j_\alpha(\kappa)}(\kappa)  = \eta_\alpha$ and
$s_{j_{\alpha,\beta}(\kappa)}(j_\alpha(\kappa)) = \eta_\beta$.
In this paper we are only interested in values $\eta_\alpha,\eta_\beta < \lambda$.

\begin{notation}\label{Notations - j0 embeddings}${}$
 \begin{enumerate}
 \item Denote $V[\Gzero]$ by $V^0$.
 \item For every $\alpha < o(\kappa) = \lambda$ and $\eta < \lambda$, let
$\jz{\alpha}{\eta} : V^0 \to M^0_{(\alpha,\eta)}$ 
denote the ultrapower induced embedding 
 $\jz{\alpha}{\eta} : V[\Gzero] \to M_\alpha[\Gzero_\alpha(\eta)] \cong \Ult(V^0,\Uz{\alpha}{\beta})$.
 \item 
 For every $\eta_\alpha,\eta_\beta < \lambda$, let 
  \[\iz{\alpha}{\eta_\alpha}{\beta}{\eta_\beta} : 
  M^0_{(\alpha,\eta_\alpha)} \to M^0_{(\alpha,\eta_\alpha),(\beta,\eta_\beta)}
  \cong \Ult \left(M^0_{(\alpha,\eta_\alpha)},
  \jz{\alpha}{\eta_\alpha}(\Uz{\beta}{\eta_\beta}) \right)\]
  denote the ultrapower of $M^0_{(\alpha,\eta_\alpha)}$ by $\jz{\alpha}{\eta_\alpha}(\Uz{\beta}{\eta_\beta})$, and let
 \[\jdz{\alpha}{\eta_\alpha}{\beta}{\eta_\beta} = \iz{\alpha}{\eta_\alpha}{\beta}{\eta_\beta} \circ \jz{\alpha}{\eta_\alpha} : V[\Gzero] \to  M^0_{(\alpha,\eta_\alpha),(\beta,\eta_\beta)}\]
\end{enumerate}
\end{notation}

\noindent The following summarizes the connections between iterated ultrapowers of $V$ and $V^0$:
 \begin{enumerate}
  \item  $\jz{\alpha}{\eta_\alpha} \uhr V = j_\alpha$,  
  \item $\iz{\alpha}{\eta_\alpha}{\beta}{\eta_\beta} \uhr M_\alpha = i_\beta^{M_\alpha}$, 
  \item  $\jdz{\alpha}{\eta_\alpha}{\beta}{\eta_\beta} \uhr V = j_{\alpha,\beta}$,
  \item $s^{ M^0_{(\alpha,\eta_\alpha),(\beta,\eta_\beta)}}_{j_\alpha(\kappa)}(\kappa) 
  = \eta_\alpha $,  \footnote{Here $s_\nu(\mu)$ denotes the value of the generic $\nu-$Sacks at $\mu$}
  \item $s^{ M^0_{(\alpha,\eta_\alpha),(\beta,\eta_\beta)}}_{j_{\alpha,\beta}(\kappa)}(j_\alpha(\kappa)) = \eta_\beta$. 
  \end{enumerate}
  
\noindent
Suppose that $\alpha < o(\kappa) = \lambda$ and $\eta < \lambda$.
The definition of $\Uz{\alpha}{\eta}$ requires the knowledge of 
$U_\alpha$ and $\Gzero$. For every $\beta > \alpha$ and $\eta' < \lambda$, 
   $U_\alpha \in M_\beta \subset M^0_{(\beta,\eta')}$ and 
  $\Gzero = \Gzero_{\beta}(\eta') \uhr (\kappa+1) \in M^0_{(\beta,\eta')}$. We conclude the following:
  \begin{corollary}\label{Corollary - I - mo of U0}
  For every $\alpha < \beta < o(\kappa)$ and $\eta,\eta' < \lambda$, 
   $\Uz{\alpha}{\eta} \mo \Uz{\beta}{\eta'}$\footnote{It is not difficult to verify that $\Uz{\alpha}{\eta} \mo \Uz{\beta}{\eta'}$ if and only if $\alpha < \beta$. }.
  \end{corollary}
  
\noindent The description of the normal measures on $\kappa$ in a $\pzero$ extension applies to all measurable cardinals $\nu < \kappa$:
Let  $\vec{U_\nu} = \la U_{\nu,\alpha} \mid \alpha < o(\nu)\ra$ be an $\mo$-increasing sequence of the normal measures on $\nu$ in the coherent sequence $\U \in V = L[\U]$.
Each normal measure $U_{\nu,\alpha} \in V$ extends to $\lambda(\nu)$ 
normal measures in $V[\Gzero\uhr (\nu + 1)]$, denoted $\{U^0_{\nu,(\alpha,\eta)} \mid \eta < \lambda(\nu)\}$.
No measures on $\nu$ are added or removed by the rest of the iteration since
$\pzero\setminus (\nu + 2)$ is $2^{(2^\nu)}-$distributive. Furthermore, we get that $U^0_{\nu,(\alpha,\eta)} \mo U^0_{\nu,(\beta,\eta')}$ for every $\alpha < \beta < o(\nu)$ and $\eta,\eta' < \lambda(\nu)$.

\subsection{The Poset $\pone$}\label{SubSection - I - pone} 

The poset 
$\pone =  \la \pone_\nu , \name{\qone_\nu} \mid \nu < \kappa\ra$ is a Magidor iteration of one-point Prikry forcings.
See \cite{Gitik-HB} for a comprehensive survey of Magidor iteration of Prikry type forcings.
One-point Prikry forcing is a simplified version of the well known Prikry forcing. The one-point version at a measurable cardinal $\nu$  chooses a single (indiscernible) ordinal $d(\nu) < \nu$, instead of a cofinal $\omega$ sequence.\\
Let $U$ be a normal measure on $\nu$.
The one-point Prikry forcing $Q(U)$ consists of elements $p \in U \cup \nu$.
For every $q,q' \in Q(U)$ we set,
\begin{enumerate}
\item $q \geq^*_{Q(U_\nu)} q'$ (i.e., $q$ is a direct extension of $q'$) if and only if 
either $q, q' \in U$ and $q\subset q'$, or $q,q' \in \nu$ and $q = q'$. 
\item $q \geq_{Q(U_\nu)} q'$ if and only if $q \geq^*_{Q(U_\nu)} q'$ or 
$q \in \nu$, $q' \in U$, and $q \in q'$.
\end{enumerate}
$(Q(U_\nu),\geq_{Q(U_\nu)}, \geq^*_{Q(U_\nu)})$ is a Prikry type forcing notion (\cite{Gitik-HB}).
\\

\noindent Let us describe the way $\pone$ is formed from certain one-point Prikry forcings. Let $\nu \leq \kappa$ be a measurable cardinal in $V$. Suppose
$\pone_\nu$ has been defined and $\vec{U}^1_\nu = \la U^1_{\nu,(\alpha,\beta)} \mid \beta < o(\nu), \alpha < \lambda(\nu)\ra$ is a given sequence of normal measures on $\nu$ in a $\pone_\nu$ generic extension of $V^0$. 

\begin{definition}[recipe for $\qone_\nu$]\label{Definition - recipe for Qone}
Let $\alpha,\beta < \lambda$ be the unique ordinals so that $\nu \in \Delta_\alpha(\beta)$ (i.e., 
$o(\nu) = \alpha$ and $s_\kappa(\nu) = \beta$). Define $\qone_\nu$ by 
\[\qone_\nu = 
\begin{cases}
    Q(U^1_{\nu,(\alpha,\beta)}) &\mbox{if }  \beta < \alpha \\
  0 - \text{the trivial forcing}  &\mbox{otherwise}.
\end{cases}\]
\end{definition}
${}$\\
\noindent {To complete the definition of $\qone_\nu$, we need to define the normal measures in $\vec{U}^1_\nu = \la U^1_{\nu,(\alpha,\beta)} \mid \beta < o(\nu), \alpha < \lambda(\nu)\ra$.
These are given in Definitions \ref{Definition - I - U^* single ultrapower} and
\ref{Definition - I - k alpha beta embedding}. Note, however,} that Definition \ref{Definition - I - U^* single ultrapower} is the only one which applies to $\qone_\nu$.

To simplify the notations, let us assume that $\nu = \kappa$ and use the abbreviations $\Uz{\alpha}{\beta}$ for $U^0_{\kappa,(\alpha,\beta)}$, and 
$\Uo{\alpha}{\beta}$ for $U^1_{\kappa,(\alpha,\beta)}$. {Our definitions make use of the embeddings} $\jz{\alpha}{\eta}$, $\iz{\alpha}{\eta_\alpha}{\beta}{\eta_\beta}$, and $\jdz{\alpha}{\eta_\alpha}{\beta}{\eta_\beta}$, introduced in \ref{Notations - j0 embeddings} above. \\

\begin{definition}[Prikry function]${}$
\begin{enumerate}
 \item Let $\delset' = \{ \nu \in \delset \mid 0_{\pone_\nu} \force \qone_\nu \text{ is not trivial }\}.$
 \item Let $\name{d} : \delset' \to \kappa$, be the $\pone$ name for the \emph{generic Prikry function}, so that 
      for every $V^0$ generic filter $\Gone \subset \pone$, $d(\nu) < \nu$ is the $\qone_\nu$ generic
      point given by $\Gone$. 
\end{enumerate}
\end{definition}
${}$

Let $\Gone \subset \pone$ be generic over $V^0$.
\begin{definition}[$\Uo{\alpha}{\beta}$ for $\alpha \geq \beta$]\label{Definition - I - U^* single ultrapower}${}$\\
Suppose that $\alpha \geq \beta$, where $\beta < o(\kappa)$ and $\alpha < \lambda(\kappa) = \lambda$. 
Let $X$ be a subset of $\kappa$ in $V^0[\Gone]$ and $\name{X}$ be a $\pone$ name of $X$.  
Set $X \in \Uo{\alpha}{\beta}$ if and only if there are 
$p \in \Gone$ and $q \geq^* \jz{\beta}{\alpha}(p) \setminus \kappa$ so that 
$p \fr q \geq^* \jz{\beta}{\alpha}(p)$ is a condition in 
$\jz{\beta}{\alpha}(\pone)$ and
\begin{equation}
p \fr q \force_{\jz{\beta}{\alpha}(\pone)} \can{\kappa} \in \jz{\beta}{\alpha}(\name{X}).
\end{equation}
\end{definition}
 \noindent 
 Definition \ref{Definition - I - U^* single ultrapower} implies that $\Uo{\alpha}{\beta}$ extends $\Uz{\beta}{\alpha}$ in $V^0$. In particular $\delset_{\beta}(\alpha) \in \Uo{\alpha}{\beta}$. 
 It is not difficult to verify that $\Uo{\alpha}{\beta}$ is a
 normal measure on $\kappa$. For proof see \cite{Gitik-HB} or
 the description of $U_0^*$ in \cite{OBN - Magidor Iteration}\footnote{to verify the normality of $\Uo{\alpha}{\beta}$, note that $\kappa \in \jz{\beta}{\alpha}(\delset_\beta(\alpha))$,
 so by Definition \ref{Definition - recipe for Qone}, stage $\kappa$ of $\jz{\beta}{\alpha}(\pone)$ is trivial}.

\begin{remark}\label{Remark p- for U^* single ultrapower}
Suppose that $X \in \Uo{\alpha}{\beta}$, $\alpha \geq \beta$, 
and let $p \in \Gone$ and $q \geq^* \jz{\beta}{\alpha}(p) \setminus \kappa$
{as in Definition \ref{Definition - I - U^* single ultrapower} above.
Essentially, one can assimilate $q$ into $\jz{\beta}{\alpha}(p)$, and use it to produce a simpler characterization for the sets in $\Uo{\alpha}{\beta}$:}
Let $Q$ be a function representing $q$ in $\Mz{\beta}{\alpha}$ where $Q(\alpha) \geq^* p \setminus \alpha$ for every $\alpha < \kappa$. 
Let $t \geq^* p$ be the condition obtained from $p$ be reducing each $p_\nu$, $\nu \not\in \supp(p)$, to $t_\nu = p_\nu \cap \Delta_{\alpha < \nu}Q(\alpha)_\nu$. It follows that for every $\alpha < \kappa$, $t^{-\alpha} \geq^* Q(\alpha)$, where $t^{-\alpha}$ is the condition obtained from $t$ by replacing each set $t_\nu$, $\nu > \alpha$, with $t_\nu \setminus \alpha+1$. \\
By a standard density argument it follows that for every $X \in \Uo{\alpha}{\beta}$, $\alpha \geq \beta$, there is some $p \in \Gone$ so that
\[\jz{\beta}{\alpha}(p)^{-\kappa} \force \can{\kappa} \in \jz{\beta}{\alpha}(\name{X}).\]
\end{remark}

\noindent We proceed to define $\Uo{\alpha}{\beta}$ when $\alpha < \beta$. We first introduce the following auxiliary definitions.  
\begin{definition}[$k^0_{\alpha,\beta}$ and $p^{+(\mu,\nu)}$]\label{Definition - I - k alpha beta embedding}${}$ 
\begin{enumerate}
 \item For $\alpha < \beta < \lambda = o(\kappa)$, let $k^0_{\alpha,\beta} : V^0 \to N^0_{\alpha,\beta}$ denote the iterated ultrapower
$\jdz{\alpha}{\beta}{\beta}{\alpha} : V^0 \to  M^0_{(\alpha,\beta),(\beta,\alpha)}$ (introduced in \ref{Notations - j0 embeddings}).

 \item For every condition $p \in \pone$, $\nu < \kappa$ so that $p\uhr \nu \force \name{p_\nu} \in Q(\name{U^*_\nu})$, and
 $\mu < \nu$, let $p^{+(\mu,\nu)}$ denote the condition obtained from $\name{p}$ by replacing $\name{p_\nu}$ with $\can{\{\mu\}}$, 
 i.e., $p^{+(\mu,\nu)} \force \can{\mu} = \name{d}(\can{\nu})$. 
\end{enumerate}
\end{definition}

\noindent Note that $p^{+(\mu,\nu)}$ is not necessarily an extension of $p$. If $p\uhr\nu \force \can{\mu} \in \name{p_\nu}$ then
$p^{+(\mu,\nu)}$ is an extension of $p$. 
\begin{definition}[$\Uo{\alpha}{\beta}$ for $\alpha < \beta$]\label{Definition - I - U^* iterated ultrapower}${}$\\
Let $\alpha < \beta < \lambda$. In $V^0[\Gone]$ define $\Uo{\alpha}{\beta}$ to be the set of all $X = \name{X}_{\Gone} \subseteq \kappa$ for which there are $p \in \Gone$ and $q \geq^*  k^0_{\alpha,\beta}(p)\setminus \kappa$ such that
 $(p \fr q)^{+(\kappa,\jz{\alpha}{\beta}(\kappa))} \geq  p \fr q$, and 
\begin{equation}
(p \fr q)^{+(\kappa,\jz{\alpha}{\beta}(\kappa))} \force \can{\kappa} \in k^0_{\alpha,\beta}(\name{X}).
\end{equation}
\end{definition}
\noindent For proof that $\Uo{\alpha}{\beta}$ is a normal measure on $\kappa$, see \cite{OBN - Magidor Iteration}\footnote{see the proof for the normality of $U_1^{\times}$.}. It follows that $\Uo{\alpha}{\beta}$ extends $\Uz{\alpha}{\beta}$ and in particular $\delset_{\alpha}(\beta) \in \Uo{\alpha}{\beta}$.

\begin{remark}\label{Remark p+- for U^* iterated ultrapower}
Similar to Remark \ref{Remark p- for U^* single ultrapower}, one can 
show that for every $X \in \Uo{\alpha}{\beta}$ there exists a condition $p \in \Gone$ so that
\[
k^0_{\alpha,\beta}(p)^{+\left( \kappa, \jz{\beta}{\gamma}(\kappa) \right)-\kappa- \jz{\beta}{\gamma}(\kappa)}\force \can{\kappa} \in k^0_{\alpha,\beta}(\name{X}).
\]
{See  \cite{OBN - Magidor Iteration} for further details.}
\end{remark}

\subsection{Separation by Sets}
Let $\Gone \subset \pone$ be a generic filter over $V^0 = V[\Gzero]$.
Let us denote $V^0[\Gone]$ by $V^1$. 
The normal measures $\{\Uo{\alpha}{\beta} \mid \alpha \leq \beta < \lambda\}$ will be used to realize $R_\lambda$.
As mentioned in the introduction, we would like the normal measure on $\kappa$ in $V^1$ to be separated by sets.  

\begin{definition}[$\Gamma, \X{\alpha}{\beta}$]
Define sets in $V^1$:
\begin{enumerate}
\item  $\Gamma = d``\delset'$, the set of Prikry generic points.
\item For every $\alpha,\beta < \lambda$,
\[\X{\alpha}{\beta} = 
\begin{cases}
  \delset_{\beta}(\alpha)\setminus \Gamma &\mbox{if }  \alpha \geq \beta \\
  \delset_{\alpha}(\beta)\cap \Gamma &\mbox{if } \alpha < \beta
\end{cases}\]
\end{enumerate}
\end{definition}

\noindent A simple inspection of Definitions \ref{Definition - I - U^* single ultrapower} and \ref{Definition - I - U^* iterated ultrapower} shows that $\X{\alpha}{\beta} \in \Uo{\alpha}{\beta}$ for every $\alpha,\beta < \lambda$. 

\begin{corollary}\label{Corollary - I - U1 separating sets}
The sets in $\{\X{\alpha}{\beta} \mid \alpha,\beta < \lambda\}$ are pairwise disjoint, 
and $\X{\alpha}{\beta} \in \Uo{\alpha}{\beta}$ for all $\alpha,\beta < \lambda$.
In particular the measures in $\{\Uo{\alpha}{\beta} \mid \alpha,\beta < \lambda\}$ are separated by sets. 
\end{corollary}

\section{The Restriction $\jo{\alpha}{\beta} \uhr V$}\label{Section - I - description of j}
For every $\alpha,\beta < \lambda$, let $\jo{\alpha}{\beta} : V^1 \to \Mo{\alpha}{\beta} \cong
\Ult(V^1,\Uo{\alpha}{\beta})$.
The purpose of this section is to describe the restriction $\jo{\alpha}{\beta} \uhr  V^0$ as an iterated ultrapower of the measures in $V^0$.
For every $\alpha,\beta < \lambda$, we first define an iterated ultrapower $\T^0$
resulting in an embedding $\pi^0_{\alpha,\beta} : V^0 \to Z^0_{\alpha,\beta}$.
The main proposition in this section (Proposition \ref{Proposition - I - identify the ultrapower restriction of Uo alpha beta}) states that $\pi^0_{\alpha,\beta} = \jo{\alpha}{\beta}\uhr V^0$. {While the definition of $\pi^0_{\alpha,\beta}$ makes the statement natural, its proof requires preliminary technical results. 
For simplicity we assume that the ground model $V$ is a Mitchell model
$V = L[\U]$ with $V = \K(V)$, where $\U$ is the coherent sequence of normal measures and $o(\kappa) = \lambda$. In particular, $V$ does not contain an overlapping extender.
Let $\alpha,\beta < o(\kappa) = \lambda$. 
Our ground model assumption $V = \K(V) = L[\U]$ and the fact 
$V^1 = V[\Gzero*\Gone]$ imply the following:
\begin{enumerate}
 \item $V = \K(V^1)$,
 \item  $\jo{\alpha}{\beta}\uhr V : V \to Z_{\alpha,\beta}$ is an iterated ultrapower of $V$,
 \item if $\jo{\alpha}{\beta} : V^1 \to M^1_{\alpha,\beta}$ then 
 $M^1_{\alpha,\beta} = Z_{\alpha,\beta}[\Gzero_{\alpha,\beta} * \Gone_{\alpha,\beta}]$ where
  $\Gzero_{\alpha,\beta} * \Gone_{\alpha,\beta}  \subset \jo{\alpha}{\beta}(\pzero * \pone)$ 
  is generic over $Z_{\alpha,\beta}$. 
\end{enumerate}
We refer to \cite{zeman} for these results.
The definition of $\pi^0_{\alpha,\beta}$ for $\alpha,\beta < \lambda$ 
makes use of the ultrapower embedding $\jz{\alpha}{\beta} : V^0 \to \Mz{\alpha}{\beta}$ defined in 
\ref{Definition - Uz alpha beta and jz alpha beta},
and the iterated ultrapower embedding $k^0_{\alpha,\beta} : V^0 \to N^0_{\alpha,\beta}$ defined in 
\ref{Definition - I - k alpha beta embedding}.
Let $\vec{\Delta} = \la \Delta_\alpha(\eta) \mid \alpha < o(\kappa), \eta < \lambda\ra$.

\begin{definition}[$\pi^0_{\alpha,\beta}$]\label{Definition - I - pi0  alpha beta}
${}$\\
$\pi^0_{\alpha,\beta}$ results from a linear iteration 
$\T^0 = \la Z_i^0, \sigma^0_{i,j} \mid 0 \leq i < j \leq \theta\ra$, 
with critical points $\nu_i = \cp(\sigma^0_{i,i+1})$ of length $\theta$.
Here $Z_i^0$ are the intermediate models (iterands) of the iteration, and $\sigma^0_{i,j} : Z_i^0 \to Z_j^0$ are the connecting
iterations. For every $i < \theta$ we denote the image of the $i-$th critical point $\nu_i$, $\sigma^0_{i,i+1}(\nu_i)$,
by $\nu_i^1$. 
We set $Z^0_0 = V^0$, $\sigma^0_{0,0} = id_{Z^0_0}$, $Z^0_1 = N^0_{\alpha,\beta}$, and
\[\sigma^0_{0,1} = 
\begin{cases} 
\jz{\beta}{\alpha} &\mbox{if }  \beta \leq \alpha \\
k^0_{\alpha,\beta} &\mbox{if } \alpha \geq \beta.
\end{cases}\]
We define $\nu_0 = \kappa$, and set $\nu^1_0 = \jz{\alpha}{\beta}(\kappa) < \sigma^0_{0,1}(\kappa)$ if $\alpha < \beta$, and 
leave $\nu^1_0$ undefined otherwise.  \\

\textbf{Successor stage:}
Suppose that $\T^0\uhr i$ has been defined up to stage $1 \leq i < \theta$, define $Z^0_{i+1}$ and $\sigma^0_{i,i+1}$ as follows:
Let $\nu^*_i$ be the supremum of $\{\nu_j \mid j < i\}$ (the set of critical points in $\T^0\uhr i$), and take $\nu_i$ to be the minimal ordinal $\nu \geq \nu_i^*$ which satisfies
\begin{enumerate}
 \item The forcing of $\sigma^0_{0,i}(\pone)$ at stage $\nu$ is not trivial, i.e. $\nu \in \sigma^0_{0,i}(\delset')$, and
 \item $\nu$ does not belong to $\sigma^0_{0,i}``\{\nu_j^1 \mid j < i\}$.
\end{enumerate}  
 These two requirements imply that the critical points of the 
 iteration  $\T^0$ are strictly increasing (i.e. the iteration is normal). 
 Since $\nu_i \in \sigma^0_{0,i}(\delset')$ then there are unique $\beta_i < \alpha_i$ so that 
 $\nu_i \in \sigma^0(\delset')_{\alpha_i}(\beta_i)$. 
 We define $\sigma^0_{i,i+1} = j^0_{\nu_i, (\beta_i,\alpha_i)} : Z_i^0 \to Z_{i+1}^0$ and 
 set $\nu_i^1 = j^0_{\nu_i, (\beta_i,\alpha_i)}(\nu_i)$. \\
 
 \textbf{Limit stage:}
 If $\delta < \theta$ is a limit ordinal then we take $Z^0_\delta$ to be the direct limit of $\T^0\uhr\delta$. \\
 ${}$\\
 The iteration terminates at stage $\theta$, when
 $\sigma^0_{0,\theta}(\delset') \subset \{\nu_i^1 \mid i < \theta \} \cup \nu^*_\theta$.
 Note that 
 \[\sigma^0_{0,1}(\kappa) =
 \begin{cases} 
 \jz{\beta}{\alpha}(\kappa) &\mbox{if }  \beta \leq \alpha \\ 
 k^0_{\alpha,\beta}(\kappa) &\mbox{if } \alpha \geq \beta.
\end{cases}\]
 By induction on $i < \theta$,  it is not difficult to verify that $\sigma^0_{0,1}(\kappa)$ is a fixed point of $\sigma^0_{1,i}$. It follows that the iteration must terminate after at most $\sigma^0_{0,1}(\kappa)$ many steps
  as each $\nu_i < \sigma^0_{1,i}(\sigma^0_{0,1}(\kappa)) = \sigma^0_{0,1}(\kappa)$, and the iteration is normal.
\end{definition}

 The following facts summarizes the main properties of $\pi^0_{\alpha,\beta}$, and can be easily proved by induction on $1 \leq i < \theta$.
 \begin{corollary}${}$
 \begin{enumerate}
 \item  For every $\alpha \geq \beta$ we have:
  \begin{itemize}
   \item $\sigma^0_{0,1} = \jz{\beta}{\alpha}$, $\nu_0 = \kappa$, and $\nu_0^1$ is not defined,
   \item for every $1 \leq i \leq \theta$ both $\kappa = \nu_0$ and $\jz{\beta}{\alpha}(\kappa)$ are not moved by 
	$\sigma^0_{1,i}$, and  $\nu_i \in (\kappa,\jz{\beta}{\alpha}(\kappa))$.
   \item $\sigma^0_{0,i}(\delset')\cap [\kappa,\nu_i) = \{\nu_j^1 \mid j < i\}$ for every $1 \leq i < \theta$. 
  \end{itemize}

   \item For every $\alpha < \beta$ we have:
  \begin{itemize}
   \item $\sigma^0_{0,1} = k^0_{\alpha,\beta}$, $\nu_0 = \kappa$, $\nu_0^1 = \jz{\alpha}{\beta} < k^0_{\alpha,\beta}(\kappa)$,
   \item for every $1 \leq i \leq \theta$, neither $\nu_0$, $\nu_0^1$, nor $k^0_{\alpha,\beta}(\kappa)$ are moved
	  by $\sigma^0_{1,i}$. Furthermore, each critical point $\nu_i$ either belongs to $(\nu_0,\nu_0^1)$ or $(\nu^1_0, k^0_{\alpha,\beta}(\kappa))$.
   \item  For every $i < \theta$ with $\nu_i \in (\nu_0,\nu_0^1)$ we have
	  $\sigma^0_{1,i}(\delset') \cap [\kappa,\nu_i) = \{ \nu_j^1 \mid 1 \leq j < i \}$.
   \item  For every $i < \theta$ with $\nu_i \in (\nu^1_0,k^0_{\alpha,\beta}(\kappa))$,
	  $\sigma^0_{1,i}(\delset') \cap [\kappa,\nu_i) = \{ \nu_j^1 \mid 0 \leq j < i \}$.  
  \end{itemize}
  \end{enumerate}
  \end{corollary}

\begin{proposition}\label{Proposition - I - identify the ultrapower restriction of Uo alpha beta} For every $\alpha,\beta < \lambda$, $\pi^0_{\alpha,\beta}$ is the restriction of 
$\jo{\alpha}{\beta} : V^1 \to \Mo{\alpha}{\beta} \cong
\Ult(V^1,\Uo{\alpha}{\beta})$ to $V^0$.
\end{proposition}

\subsection{Structural results for dense open sets in $\pone$}\label{SubSection - I - Structural Results On Dense Open Sets}

In this section we prove several preliminary results, which will be used in the proof of Proposition \ref{Proposition - I - identify the ultrapower restriction of Uo alpha beta}.
We focus on a specific family of finite subiterations of $\T^0$ named \emph{structural iterations for} $\Uo{\alpha}{\beta}$, and use them to describe a criterion for meeting dense open sets in $\pi^0_{\alpha,\beta}(\pone)$.

\begin{definition}[structural function, and structural extension]\label{Definition - I - Structural Function and Extension}${}$\\
We define by induction on $n < \omega$ a \textit{structural function} $f$ of degree $n$, avoiding $b \subset \kappa$.
For $n = 0$, a structural function of degree $0$ is the trivial function $f^0 = \emptyset$.
A function $f = f^{n+1}$ is a structural function of degree $n+1$, avoiding $b$, 
if there is a unique ordinal $\nu_f < \kappa$, and a $\pone_{\nu_f}$ name $\name{X_f}$ so that the following holds:
\begin{enumerate}
 \item $\nu_f \in \delset' \setminus b$. 
 \item $0_{\pone_{\nu_f}} \force \name{X_f} \in \name{U_{\nu_f}^*}$.
 \item $\dom(f)$ is the set of all $\pone_{\nu_f}$ names for ordinals in $\name{X_f}$.
 \item For every name $\tau \in \dom(f)$, $f(\tau)$ is a structural function $g$ of degree $n$ avoiding $b$, 
	and $\nu_g < \nu_f$. 
\end{enumerate}
We say that $f$ is a structural function if there exists some $n < \omega$ so that $f$ is a structural function of degree $n$.\\
Let $p$ be a condition, and $f$ be a structural function avoiding $\supp(p)$. 
We say that a condition $q$ is a \textit{structural extension} of $p$ by $f$ 
if the following holds:
\begin{enumerate}
 \item If $f$ has degree $0$ then $q$ is a structural extension of $p$ by $f$ if $q \geq^* p$.
 \item If $f$ has degree $n+1$, then $q$ is a structural extension of $p$ by $f$ if 
	there are $r \geq^* p\uhr \nu_f$ and $\tau \in \dom(f)$ so that $r \force \tau \in p_{\nu_f}$, 
	and $q$ is a structural extension of $r \fr (p \setminus \nu_f)^{+(\tau,\nu_f)}$ by (the degree $n$ structural function) $f(\tau)$. Note that $r \fr (p \setminus \nu_f)^{+(\tau,\nu_f)} \geq p$.
\end{enumerate}
\end{definition}

\begin{lemma}\label{Lemma - I - Structural Lemma}
For every open dense set $D \subset \pone$ and $p \in \pone$ there exists a structural function $f$,
avoiding $\supp(p)$, so that every structural extension of $p$ by $f$ has a direct extension in $D$.
\end{lemma}

\begin{proof}
We prove by induction on $\nu \leq \kappa$, that the above holds for every dense open set $D \subset \pone_\nu$ and $p \in \pone_\nu$. 
Suppose the claim holds for every dense open set $D \subset \pone_\nu$ and $p \in \pone_\nu$. We have 
$\pone_{\nu+1} = \pone_\nu * \name{Q_\nu}$.
If $\nu \not\in \delset'$ then $\name{Q_\nu}$ is trivial and there is nothing to prove.]
If $\nu \in \delset'$ then $\name{Q_\nu} = Q(\name{U_\nu^*})$. 
Let $D \subset \pone_{\nu+1}$ be a open dense set, and $p = p \uhr \nu \fr p_\nu \in \pone_{\nu+1}$. 
If $\nu \in \supp(p)$ then the forcing $\pone_\nu$ over $p$ is equivalent to $\pone_\nu$. \\
For every $\nu \not\in \supp(p)$ the name $p_\nu$ is a $\pone_\nu$ name of a set in $U_\nu^*$.
For every $G_\nu \subset \pone_\nu$ generic, the set $D(G_\nu) = \{(q_\nu)_{G_\nu} \mid q \in D\}$ is dense open set in $Q(U^*_\nu)$.
It follows there is some $Y_\nu \subset (p_\nu)_{G_\nu}$ with $Y_\nu \in U_\nu^*$ so that
 $\mu \in D(G_\nu)$ for every $\mu \in Y_\nu$.
Let $\name{Y_\nu}$ be a name for $Y_\nu$ in $\pone_\nu$. For every name $\tau$ of an ordinal in $\name{Y_\nu}$, the 
set $D_\tau = \{q \geq p\uhr \nu \mid q \fr \la \tau \ra \in D\}$ is dense open in $\pone_\nu$.
The inductive assumption guarantees there is some $n(\tau) < \omega$ and a structural function $f(\tau)$ of degree $n(\tau)$,
such that every structural extension of $p\uhr \nu$ by $f(\tau)$, has a direct extension in $D_\tau$. 
For every $n < \omega$, let $\name{X^n_\nu}$ be the $\pone_\nu$ name of the 
set $\{\tau \in Y_\nu \mid n(\tau) = n\}$, and 
let $\sigma^0_n$  be the $\pone_\nu$ statement
\[ \sigma^0_n :  \name{X^n_\nu} \in U_\nu^*. \]
Since $\pone_\nu$ satisfies the Prikry condition, there exists a unique $n < \omega$ and some $r \geq^* p\uhr\nu$ so that
$r \force \name{X_\nu^n} \in \name{U_\nu^*}$.  Let us denote $X_\nu^n$ by $X_\nu$. 
We conclude that for every name of an ordinal in $X_\nu$, $\tau$, there is a structural tree $f(\tau)$ of degree $n$
such that every $f(\tau)$ structural extension of $r^{+(\tau,\nu)}\ra$, has a direct extension in $D$. 
It follows that the function $f$ mapping every such name $\tau$ to $f(\tau)$, is a structural function of degree $n+1$, 
and the claim of the Lemma holds with respect to $p,f$.\\
Let $\delta \leq \kappa$ be a limit ordinal, and suppose that the claim holds in every $\pone_\nu$ for $\nu < \delta$.
Fix some $p \in \pone_\delta$ and a dense open $D \subset \pone_\delta$.
In order to prove the result it is sufficient to show that for some $\nu < \delta$ and a $\pone_\nu$ name $\name{t}$, $p\uhr\nu \force \name{t} \geq^* p\setminus \nu$ and 
$D_{\name{t}} = \{r \geq p\uhr\nu \mid r \fr \name{t} \in D\} \subset \pone_\nu$ is dense open.
Suppose otherwise, and let us construct a direct extension of $p$, $p^* = \la p^*_\nu \mid \nu < \delta\ra$ so that for every $\nu < \delta$, $p^*\uhr\nu \force \sigma^0_\nu$ where
\begin{equation}\label{Equation - I - Inductive assumption in Structural Lemma}
 \sigma^0_\nu : \forall t \geq^* (p\setminus \nu). \thinspace  t \not\in D(\name{G_\nu}).
\end{equation}
Note that the existence of $p^*$ would contradict the fact that $D$ is dense.
Suppose $p^* \uhr \nu = \la p^*_\mu \mid \mu < \nu \ra$ has been defined and satisfies
\ref{Equation - I - Inductive assumption in Structural Lemma}.
Fix a $\pone_\nu$ generic $G_\nu$ with $p^*\uhr \nu \fr p\setminus \nu \in G_\nu$,
and consider the forcing $\pone\setminus \nu = \qone_{\nu+1} * \pone\setminus(\nu+1)$. 
Since $\qone_{\nu+1}$ satisfies the Prikry property, there exists some $r_\nu \geq^* (p_\nu)_{G_\nu}$ which decides
$\sigma^0_{\nu+1}$ \footnote{considered {as} a $\qone_\nu-$statement.}.
If $r_\nu \force \neg\sigma^0_{\nu+1}$ there would be $q_{>\nu} \geq^* p\setminus (\nu+1)$ and $r^*_\nu \geq^* r_\nu$ so that
$r^*_\nu \fr q_{>\nu} \in D(G_\nu)$. 
This is impossible as $r^*_\nu \fr q_{>\nu} \geq^* p\setminus \nu$ while $\sigma^0_\nu$ holds in $V[G_\nu]$ as $p^*\uhr \nu \in G_\nu$. \\
We conclude $r_\nu$ forces $\sigma^0_{\nu+1}$ in $V[G_\nu]$. Back in $V$, let $p^*_\nu$ be a $\pone_\nu$ name for $r_\nu$, so
$p^*\uhr \nu \force p^*_\nu \geq^* p_\nu$, and $p^*\uhr\nu \fr p^*_\nu \force \sigma^0_{\nu+1}$. \\
Let $\delta' < \delta$ be a limit ordinal. Suppose $p^* \uhr \delta' = \la p^*_\mu \mid \mu < \delta'\ra$ has been constructed, and let us
show that $p^*\uhr \delta' \force \sigma^0_{\delta'}$. Otherwise, there would be conditions, $t \in \pone\setminus \delta'$ and 
$r \geq p^*\uhr\delta'$, so that $p^*\uhr\delta' \force t \geq^* p\setminus \delta'$ and $r \force t \in D(\name{G_{\delta'}})$.
Thus there exists some $r' \geq r$ so that $r' \fr t \in D$.
As $\delta'$ is a limit ordinal and $\supp(r')$ is finite it follows there is some $\nu < \delta'$ with
$r' \setminus \nu \geq^* (p^*\uhr\delta'\setminus \nu)$.
Let $s = r' \uhr \nu$, we get $(r'\setminus \nu) \fr t \geq^* p\setminus \nu$ and 
$s \force (r'\setminus \nu) \fr t \in D(\name{G_\nu})$.
This is absurd as $s \geq p^*\uhr\nu$ and must force $\sigma^0_\nu$.
\end{proof}

\begin{definition}[structural iteration and compatible conditions]\label{Definition - I - Structural Iteration and Compatible}${}$\\
For $\alpha,\beta < \lambda$  a \textit{structural iteration} 
for $\Uo{\alpha}{\beta}$ is a finite iterated ultrapower 
$\vec{M} =\la M_m , j_{k,m} \mid k < m  \leq n\ra$ of length $n < \omega$,
satisfying the following properties: \\
$M_0 = V^0$, 
\[j_{0,1} = 
\begin{cases}
\jz{\beta}{\alpha} : V^0 \to M^0_{(\beta,\alpha)} &\mbox{if }  \beta \leq \alpha \\
k^0_{\alpha,\beta} : V^0 \to M^0_{(\alpha,\beta),(\beta,\alpha)} &\mbox{if }  \alpha > \beta 
\end{cases}.\]
\noindent Define $\nu_0 = \kappa$, set 
$\nu_0^1 = j^0_{(\alpha,\beta)}(\kappa) < k^0_{\alpha,\beta}(\kappa)$ if $\alpha < \beta$, 
and leave $\nu_0^1$ undefined otherwise. 
For every $1 \leq k < n$, suppose that  $\vec{M}\uhr {k+1} = \la M_m , j_{i,m} \mid i < m \leq k\ra$ 
and $\la  (\nu_i, \nu_i^1) \mid i < k\ra$ have been defined. Then there is an ordinal $\nu_k < j_{0,k}(\kappa)$ so $\nu_k  \in j_{0,k}(\delset')\setminus (\kappa \cup \{\nu_i^1 \mid i < k \})$ and 
unique $\alpha_k,\beta_k$ with $\nu \in j_{0,k}(\delset)_{\alpha_k}(\beta_k)$\footnote{note that we must have $\alpha_k < \beta_k$ since $\nu \in j_{0,k}(\delset')$.}.
\begin{enumerate}
 \item $j_{k,k+1} : M_k \to M_{k+1} \cong \Ult(M_k,U^0_{\nu_k,{(\beta_k,\alpha_k)}})$,
  \item $\nu^1_k = j_{k,k+1}(\nu_k)$. 
\end{enumerate}
We say that a condition $p \in j_{0,n}(\pone)$ is \textit{compatible} with the structural iteration
$\vec{M}$, if there exists a sequence $\vec{p} = \la p^k \mid k \leq n\ra$ with the following properties:
\begin{enumerate}
 \item $p^0 \in \pone$ in $V^0$.
 \item 
 $p^1 = 
\begin{cases}
 j_{0,1}(p^0) = \jz{\beta}{\alpha}(p^0) &\mbox{if }  \beta \leq \alpha \\
 
  (p \fr q)^{+(\kappa,\jz{\alpha}{\beta}(\kappa))} \text{ as in Definition \ref{Definition - I - U^* iterated ultrapower} }
					 &\mbox{if }  \alpha < \beta.
\end{cases}$
 
  \item For every $1 \leq k < n$,  
  \[p^{k+1} = (p \fr q \fr (j_{k,k+1}(p)\setminus \nu^1_k))^{+(\nu_k,\nu_{k}^1)},\] 
  where
  \begin{itemize}
   \item $q \geq^* j_{k,k+1}(p)\uhr [\nu_k,\nu_k^1)$,
   \item $p \fr q \force \nu_k \in j_{k,k+1}(p_{\nu_k})= j_{k,k+1}(p)_{\nu_k^1}$. 
	  Note that the existence of such $q$ is guaranteed by Definition \ref{Definition - I - U^* single ultrapower}.
  \end{itemize}
  \item $p \geq^* p^n$.
\end{enumerate}
\end{definition}

Comparing the last definition with the definition of the iteration $\T^0$ for $\Uo{\alpha}{\beta}$, it clear that structural iterations are all finite subiterations of $\T^0$ and that the $\T^0-$resulting limit,
$\pi^0_{\alpha,\beta} : V^0 \to Z^0_{\alpha,\beta}$, is also the limit of the directed system of all structural iterations for $\Uo{\alpha}{\beta}$.\\

Before proceeding, we point out the following simple facts:
\begin{remarks}\label{Remark - III - More Structural Remarks}
Let $\vec{M} =\la M_m , j_{k,m} \mid k < m  \leq n\ra$ be a structural iteration for $\Uo{\alpha}{\beta}$.
\begin{enumerate}
 \item 
The embedding $j_{0,1}$ coincides with the ultrapower embedding used for the definition of $\Uo{\alpha}{\beta}$ in  \ref{Definition - I - U^* single ultrapower}  and \ref{Definition - I - U^* iterated ultrapower} for $\Uo{\alpha}{\beta}$.

 \item  For every $k < n - 1$, $\nu_k,\nu^1_k$ are not moved by $j_{k+1,n}$.
 \item For every $k$, $1 \leq k < n$,  $\nu_k \in j_{0,k}(\delset')$ implies that
 $\nu_k^1 \in j_{0,k+1}(\delset')$. Furthermore since $\Uz{\beta_k}{\alpha_k}$ (the measure generating $j_{k,k+1}$) does not include $j_{0,k}(\delset') \cap \nu_k$ we get that $\nu_k \not\in j_{0,k+1}(\delset')$. Similarly, we have that  
 $\kappa \not\in k^0_{\alpha,\beta}(\delset')$ when $\alpha < \beta$, and $\kappa \not\in  \jz{\beta}{\alpha}(\delset')$ when $\alpha \geq \beta$. Therefore $\nu_k \not\in j_{0,n}(\delset')$ for every $k < n$. 
 \item If $p \in j_{0,n}(\pone)$ is compatible with the iteration $\vec{M}$ which is witnessed by a sequence $\la p^i \mid i \leq n\ra$ then $p^{k+1} \force j_{0,k+1}(\name{d})(\nu^1_k) = \nu_k$  for every $k < n$.
	Both $\nu_k,\nu^1_k$ are not moved by the rest of the iteration hence
	\[ p \force j_{0,n}(\name{d})(\nu^1_k) = \nu_k \]
	whenever $\nu^1_k$ is defined.	
\item Suppose that $p_0,p_1 \in j_{0,n}(\pone)$ are compatible with $\vec{M}$. Let $\la p^k_0 \mid k\leq n\ra$, $\la p^k_1 \mid k \leq n\ra$ witnessing sequences for $p_0$, $p_1$ respectively. 
It is easy to see by induction on $k \leq n$ that if 
$p^0_0,p^0_1 \in \pone$ are compatible in $\pone$ then 
 $p^k_0,p^k_1$ are compatible in $j_{0,k}(\pone)$. 
\end{enumerate}
\end{remarks}

\begin{lemma}\label{Lemma - I - Structural Iteration and Structural Function}
Let $\vec{M_0}  =\la M_m , j_{k,m} \mid k < m  \leq n_0\ra$ be a structural iteration of length $n_0$ and 
$p \in j_{0,n_0}(\pone)$ compatible with $\vec{M_0}$. 
We have that for every structural function $f = f^n$ of degree $n$ which avoids $\supp(p) \cup \kappa$, 
there exists a structural iteration $\vec{M}  = \la M_m , j_{k,m} \mid k < m  \leq n_0 + n\ra$ of length $n_0 + n$, 
extending $\vec{M_0}$, and  $q \in j_{0,n_0+n}(\pone)$ so that
\begin{enumerate}
 \item $q$ is compatible with $\vec{M}$, and
 \item $q$ is a structural extension of $j_{n_0,n_0 + n}(p)$ by $j_{n_0,n_0 + n}(f)$.
\end{enumerate}
\end{lemma}
\begin{proof}
Let us denote $n_0 + n$ by $n^*$.
For every $k$, $n_0 \leq k < n^*$, we chose
$j_{k,k+1} :M_k \to M_{k+1}$, $p^{k+1} \in j_{0,k+1}(\pone)$, and $g^{n^*-k}$ of degree $n^*-k$,
so that $p^{k+1}$ is compatible with 
$\vec{M}\uhr (k+1)$, and $g^{n^*-k}$ avoids $\supp(p^{k+1}) \cup \kappa$.
Let $M_{n_0}$ be the last model in $\vec{M_0}$, $p^{n_0} = p$, and $g^{n^* - n_0} = g^n = f$.
Suppose that $\vec{M} \uhr k+1$, $p^k$, $g^{n^* -k}$ have been defined. 
Note that $\{\nu_i^1 \mid i < k\} \subset \supp(p^{k})$ and 
$\nu_k \in j_{0,k}(\delset') \setminus (\kappa \cup \supp(p^{k}))$.
Suppose that $\nu_k \in \delset_{\alpha_k}(\beta_k)$, and let 
$j_{k,k+1} : M_k \to M_{k+1} \cong \Ult(M_k,U^0_{\nu_k,(\beta_k,\alpha_k)})$
Let $X_{g^{n^* -k}}$ be the name associated with $g^{n^* -k}$, and let
$Y_k$ be a name for $X_{g^{n^* -k}} \cap p^{k}_{\nu_k}$.
Since $Y_k$ is a name of a set in $\Uo{\alpha_k}{\beta_k}$ then
by the definition of $\Uo{\alpha_k}{\beta_k}$ there is some $q \geq^* j_{k,k+1}(p^{k})\uhr [\nu_k,\nu^1_k)$, so that
$p^{k} \fr q \force \can{\nu_k} \in j_{k,k+1}(Y_k)$. 
Define $p^{k+1} =  \left(p^{k} \fr q \fr (j_{k,k+1}(p^k)\setminus \nu^1_k) \right)^{+(\nu_k,\nu^1_k)}$. \\
Let $\name{\nu_k}$ for an ordinal in $Y_k$ which is interpreted as $\can{\nu_k}$ by every condition which forces $\can{\nu_k} \in j_{k,k+1}(X_{g^{n^* -k}})$.
Clearly $\name{\nu_k} \in \dom(j_{k,k+1}(g^{n^* -k}))$ so we can define $g^{n^* - (k+1)} = j_{k,k+1}(g^{n^* -k})(\name{\nu_k})$.
The inductive hypothesis implies that $p^{k+1}$ is compatible with $\vec{M}\uhr k+2$ and that $g^{n^* - (k+1)}$ avoids $\kappa \cup \supp(p^{k+1})$. \\ 
The construction terminates after $n = n^* - n_0$ steps. We obtain an iteration  
$\vec{M}  =\la M_m , j_{k,m} \mid k < m  \leq n^* \ra$, $\la p^i \mid i \leq n^*\ra$ and  a structural function of degree $0$,
$g^{0} = g^{n^* - n^*} = \emptyset$.
For every $k \leq n^*$, $\nu_k,\nu^1_k$ are fixed by $j_{k+1,n}$, 
hence $g^0$ can also be described as follows:
\begin{enumerate}
 \item $h^{0} = j_{n_0,n^*}(f)$ is of degree $n$,
 \item $h^{i+1} = h^i(\name{\nu_{n_0 + i}})$ is of degree $n - i -1$ for all $i < n$,
 \item $g^0 = h^n$. 
\end{enumerate}
We conclude that the sequence
$\la j_{k,n^*}(p^{k}) \in j_{n_0,n^*}(\pone) \ra$ is a witness for the fact that
$q = p^{n^*} = j_{n^*,n^*}(p^{n^*})$ is a structural extension of $j_{n_0,n_0 + n}(p)$ by $j_{n_0,n_0 + n}(f)$. 
\end{proof}

\noindent The following concludes the findings of Lemmata \ref{Lemma - I - Structural Lemma} and \ref{Lemma - I - Structural Iteration and Structural Function}.
\begin{corollary}\label{Corollary - I - Main Structural}
Let $\vec{M_0}$ be a structural iteration of length $n_0$, and
$D$ be a $\pone-$name of a dense open set in $j_{0,n_0}(\pone)\setminus \kappa$.
For every $\vec{M_0}$ compatible condition $p \in j_{0,n_0}(\pone)$ 
there is a structural iteration $\vec{M}$ extending $\vec{M_0}$
and a $\vec{M}$ compatible condition $q \in j_{0,n^*}(\pone)$ so that 
\begin{enumerate}
 \item $q \geq j_{n_0,n^*}(p)$ (here $n^* = |\vec{M}|$),
 \item $q\uhr\kappa = j_{n_0,n^*}(p)\uhr\kappa = p\uhr\kappa$, and
 \item $q \uhr \kappa \force (q \setminus \kappa) \in D$. 
\end{enumerate}
\end{corollary}

\subsection{A proof for Proposition \ref{Proposition - I - identify the ultrapower restriction of Uo alpha beta}}
Let $\alpha,\beta < \lambda$ and $\T^0 = \la Z^0_i, \sigma^0_{i,j} \mid 0 \leq i < j < \theta\ra$ be the iteration associated
with $\Uo{\alpha}{\beta}$ (Definition \ref{Definition - I - pi0  alpha beta}). and let $\pi^0_{\alpha,\beta} : V^0 \to Z^0_{\alpha,\beta}$ be the resulting  elementary embedding. $\pi^0_{\alpha,\beta}$ is also the limit of the directed system which consists of all structural iterations of $\Uo{\alpha}{\beta}$.

 \begin{definition}${}$
 \begin{enumerate}
  \item  For every structural iteration  $\vec{M} = \la M_k, j_{k,m} \mid k \leq m \leq n\ra$ with respect to $\Uo{\alpha}{\beta}$,
 let $j_{\vec{M}} : V^0 \to M_{n}$ denote $j_{0,n}$ and $k_{\vec{M}} : M_{n} \to Z^0_{\alpha,\beta}$ denote 
 the direct limit embedding of $M_{n}$ in $Z^0_{\alpha,\beta}$. \
 \item 
 Let $\Gone \subset \pone$ be a $V^0$ generic. We say that a condition $p \in j_{\vec{M}}(\pone)$ is compatible with 
 both $\vec{M}$ and $\Gone$ if there is a witnessing sequence $\la p^k \mid k \leq n\ra$ so that $p^0 \in \Gone$. 
 \item Let $F_{\vec{M},\Gone} \subset j_{\vec{M}}(\pone)$
 be the set of all the conditions  $p \in j_{\vec{M}}(\pone)$ which are compatible with $\vec{M}$ and $\Gone$. 
 \end{enumerate}
 \end{definition}
 
\begin{proof}[(Proof. \textbf{Proposition \ref{Proposition - I - identify the ultrapower restriction of Uo alpha beta}})]${}$
Define $\Gone_{\alpha,\beta} \subset \pi^0_{\alpha,\beta}(\pone)$,
\[ \Gone_{\alpha,\beta} = \bigcup \{ k_{\vec{M}}``F_{\vec{M},\Gone} \mid \vec{M} \text{ is a structural iteration }\}. \]
It is clear $\pi^0_{\alpha,\beta}``\Gone \subset \Gone_{\alpha,\beta}$.
Furthermore, the last remark in \ref{Remark - III - More Structural Remarks} implies that every two conditions in $\Gone_{\alpha,\beta}$ are compatible. \\
We show that $\Gone_{\alpha,\beta}$ is generic over $Z^0_{\alpha,\beta}$.
Clearly, $\Gone_{\alpha,\beta}\uhr \kappa = \Gone$ is $\pone = \pi^0_{\alpha,\beta}(\pone)\uhr\kappa$
generic over $Z^0_{\alpha,\beta}$.
Let $D'$ be a $\pone$ name for a dense open set in $\pi^0_{\alpha,\beta}(\pone)\setminus \kappa$. 
Let $\vec{M_0}$ be a structural iteration for which there is $D \subset j_{\vec{M_0}}(\pone)$ such that $k_{\vec{M}}(D) = D'$.
Fix a condition $p \in F_{\vec{M_0},\Gone}$. 
By Corollary \ref{Corollary - I - Main Structural}, there is a structural iteration $\vec{M}$ extending
$\vec{M_0}$ and a compatible condition $q \in j_{\vec{M}}(\pone)$, so that
$q\uhr\kappa = p\uhr\kappa$ and $q\uhr\kappa \force q\setminus\kappa \in D$.
This implies that $q \in F_{\vec{M},\Gone}$, which in turn, implies that
$q' = k_{\vec{M}}(q) \in D' \cap \Gone_{\alpha,\beta}$. 
It follows $\Gone_{\alpha,\beta} \cap D \neq \emptyset$. \\
We can therefore extend $\pi^0_{\alpha,\beta} : V^0 \to Z^0_{\alpha,\beta}$
to an elementary embedding $\po^1_{\alpha,\beta} : V^0[\Gone] \to Z^0_{\alpha,\beta}[\Gone_{\alpha,\beta}]$ so that for every set
$x = (\name{x})_{\Gone}$,
$\pi^1_{\alpha,\beta}(x) = (\pi^0_{\alpha,\beta}(\name{x}))_{\Gone_{\alpha,\beta}}$.\\
Let $U$ denote the the normal measure on $\kappa$ in $V^0[\Gone]$ defined
by $X \in U$ if $\kappa \in \pi^1_{\alpha,\beta}(X)$. We first show that $U = \Uo{\alpha}{\beta}$, and then prove that 
$\pi^1_{\alpha,\beta}$ coincides with the ultrapower embedding of $U$.
Let $X \in \Uo{\alpha}{\beta}$. By the definition of $\Uo{\alpha}{\beta}$
(\ref{Definition - I - U^* single ultrapower} and \ref{Definition - I - U^* iterated ultrapower}) there is a $\Gone$ name $\name{X}$ for $X$ and a $j_{9,1}$ compatible condition $t \in j_{0,1}``\Gone$  so that 
$t \force_{j_{0,1}(\pone)} \can{\kappa} \in j_{0,1}(\name{X})$.
It follows that $t$ is compatible with $\Gone$, i.e., $t \in F_{Z^0_{1},\Gone}$. Let $k : Z^0_1 \to Z^0_{\alpha,\beta}$ be the direct limit embedding of the iteration. We get that $k(t) \in \Gone_{\alpha,\beta}$.
As $k$ does not move $\kappa$ (the iteration $\T$ after $Z^0_1$ is above $\kappa$) it follows that $k(t) \force \can{\kappa} \in \pi^0_{\alpha,\beta}(\name{X})$ thus $X \in U$. We conclude that $\Uo{\alpha}{\beta} \subseteq U$. $U = \Uo{\alpha}{\beta}$ as both are ultrafilters. \\

\noindent It follows the ultrapower embedding of $V^1$ by $U$ is 
$j^1_{(\alpha,\beta)} : V^1 \to M^1_{(\alpha,\beta)} \cong \Ult(V^1,\Uo{\alpha}{\beta})$ and that $\pi^1_{\alpha,\beta}$ can be factored into
$e^1_{\alpha,\beta} \circ j^1_{(\alpha,\beta)}$, where
$e^1_{\alpha,\beta} : M^1_{\alpha,\beta} \to Z^0_{\alpha,\beta}[\Gone_{\alpha,\beta}]$ maps $[f]_{\Uo{\alpha}{\beta}}$ to  $\pi^1_{\alpha,\beta}(f)(\kappa)$.
Therefore in order to show $(j^1_{(\alpha,\beta)},M^1_{\alpha,\beta}) = (\pi^1_{\alpha,\beta},Z^0_{\alpha,\beta}[\Gone_{\alpha,\beta}])$ it suffices to prove $e^1_{\alpha,\beta}$ is surjective. 
Suppose $x \in Z^0_{\alpha,\beta}[\Gone_{\alpha,\beta}]$ and let $\name{x}$ be a 
$\pi^1_{\alpha,\beta}(\pone)$ name for $x$,i.e., 
$x = (\name{x})_{\Gone_{\alpha,\beta}}$. Since $\Gone_{\alpha,\beta} = \pi^1_{\alpha,\beta}(\Gone) \in \rng(e^1_{\alpha,\beta})$. Let us show $\name{x} \in \rng(e^1_{\alpha,\beta})$.
To this end, $\name{x} \in Z^0_{\alpha,\beta}$ implies 
there is a structural iteration $\vec{M}$ and a $j_{\vec{M}}(\pone)$ name $\name{y}$ so that
$\name{x} = k_{\vec{M}}(\name{y})$. 
Let $\la \nu_k \mid k < n\ra$ be the list of critical points of 
$\vec{M}  =\la M_k , j_{k,m} \mid k < m  \leq n \ra$. 
Thus $\name{y} = j_{0,n}(h)(\nu_0,\dots,\nu_{n-1})$ for some $h : \kappa^n \to \in V^0$ in $V^0$.
For every $k < n$ let $i_k < \theta$ such that $\nu_{i_k} = k_{\vec{M}}(\nu_k)$.
By applying $k_{\vec{M}}$ we get $\name{x} = \pi^0_{\alpha,\beta}(h)(\nu_{i_0},\dots,\nu_{i_k})$.
It remains to show $\nu_{i_m} \in \rng(e^1_{\alpha,\beta})$ for each 
$m <n$. This is proved by induction.
The case $m = 0$ is trivial as $\nu_0 = \kappa$ and $k_{\vec{M}}(\kappa) = \kappa$. 
Let $0 < m < n$ and suppose that the claim holds for every $m' < m$. $\nu_m$ is the critical point of the $m-$stage of the iteration $\vec{M}$. As a member of $M_m$ (the $m-$th iterand in $\vec{M}$) we can write
$\nu_m = j_{0,m}(h)(\nu_0,\dots,\nu_{m-1})$ in $M_m$, where
$h : \kappa^{m} \to V^0$ belongs to $V^0$.
By applying $j_{m,m+1}$ we get
$\nu^1_m = j_{0,m+1}(h)(\nu_0,\dots,\nu_{m-1})$ in $M_{m+1}$. As
$j_{m+1,n}$ does not move $\nu_0,\dots,\nu_{m-1},\nu_m,\nu^1_m$,  
$\nu^1_m = j_{0,n}(h)(\nu_0,\dots,\nu_{m-1})$ in $M_n$.
Now for every $p \in F_{\vec{M},\Gone}$ we have
$p \force \can{\nu_m} = j_{0,n}(\name{d})(\can{\nu^1_m}) = j_{0,m}(\name{h'})(\nu_0,\dots,\nu_{m-1})$, 
where $\name{d}$ is the name of the generic Prikry function and $h' = d \circ h$.
It follows that if $q = k_{\vec{M}}(p)$ then $q \in \Gone_{\alpha,\beta}$ and
$q \force \can{\nu_{i_m}} = \pi^0_{\alpha,\beta}(h')(\nu_{i_0},\dots,\nu_{i_{m-1}})$. 
The result is therefore a consequence of the inductive assumption for $\nu_{i_0},\dots,\nu_{i_{m-1}}$.
\end{proof}

\begin{corollary}[$j^1_{\alpha,\beta}\uhr V$]\label{Corollary - I - Restriction of pi0 to V}
The restriction $j^1_{\alpha,\beta}\uhr V$ results from the following iteration $\T = \la Z_i, \sigma_{i,j} \mid 0 \leq i < j < \theta\ra$:
\begin{enumerate}
 \item  $Z_0 = V$. For $\sigma_{0,1} = \sigma^0_{0,1}\uhr V : Z_0 \to Z_1$, i.e., 
     \[\sigma_{0,1} = 
      \begin{cases}       
	\jz{\beta}{\alpha}\uhr V = j_\beta &\mbox{if }  \alpha \geq \beta \\
	k^0_{\alpha,\beta} \uhr V =  \jdz{\alpha}{\beta}{\beta}{\alpha} \uhr V = j_{\alpha,\beta}
		&\mbox{if }  \alpha > \beta 
      \end{cases}\]
      
\item Given $\T\uhr i$ and $\sigma_{j,i} : Z_j \to Z_i$ for $j < i$, so that
      $Z_j = \K(Z^0_j)$ (i.e., the core of $Z_j$), and $\sigma_{j,i} = \sigma^0_{j,i}\uhr Z_j$,
      we have 
      \[\sigma_{i,i+1} = j^0_{\nu_i, (\beta_i,\alpha_i)}\uhr Z_i = j_{\nu_i,\beta_i} : Z_i \to Z_{i+1}.\]
\end{enumerate}
\end{corollary}

\section{The Structure of $\mo(\kappa)$  in $V^1$}\label{Section - I - mo structure}
We prove that in $V^1$, the restriction of $\mo$ to $\{\Uo{\alpha}{\beta} \mid \alpha \leq \beta < \lambda\}$ is isomorphic to $R_\lambda$.  
\begin{proposition}\label{Proposition - MO structure on U1}
Suppose that $\alpha' \leq \beta' $, $\alpha \leq \beta$ are ordinals below $\lambda = o(\kappa)$.
In $V^0[\Gone]$, $\Uo{\alpha'}{\beta'} \mo \Uo{\alpha}{\beta}$  if and only if
$\beta' < \alpha$. 
\end{proposition}

Note that while $V^1$ contains normal measures $\Uo{\alpha}{\beta}$ for $\alpha > \beta$, Proposition \ref{Proposition - MO structure on U1} refers only to $\Uo{\alpha}{\beta}$ for $\alpha \leq \beta < \lambda$.
These additional measures are omitted from Proposition \ref{Proposition - MO structure on U1} since they do not add any essential structure to $\mo(\kappa)^{V^1}$.
More precisely, the proof of Proposition \ref{Proposition - MO structure on U1}
shows that for every $\alpha > \beta$, $\Uo{\alpha}{\beta}$ is 
equivalent to $\Uo{\alpha}{\alpha}$ in the order $\mo(\kappa)\uhr \vec{U^1}$, 
where $\vec{U^1} = \{ \Uo{\alpha}{\beta} \mid \alpha,\beta < o(\kappa)\}$.
As Section \ref{Section - I - identify normal measures} proves that the measures in $\vec{U^1}$ are all the normal measures on $\kappa$ in $V^1$, we conclude that $R_\lambda$ is isomorphic to the reduction of $\mo(\kappa)^{V^1}$. 
We separate the proof of proposition \ref{Proposition - MO structure on U1} to ``if'' and ``only if'' claims.

\begin{claim}
If $\beta' < \alpha$ then $\Uo{\alpha'}{\beta'} \mo \Uo{\alpha}{\beta}$.
\end{claim}
\begin{proof}
Suppose that $\alpha' \geq \beta'$. It is clear from Definitions \ref{Definition - I - U^* single ultrapower} and
\ref{Definition - I - U^* iterated ultrapower} that
$\Uo{\alpha'}{\beta'}$ is defined in every inner model of $V^1$
which contains $\Gzero$,$\Gone$, and $\Uz{\beta'}{\alpha'}$. 
Similarly, if $\alpha' < \beta'$ then $\Uo{\alpha'}{\beta'}$ is defined in every inner model containing $\Gzero$,$\Gone$, $\Uz{\beta'}{\alpha'}$, and $\Uz{\alpha'}{\beta'}$. \\
For every $\alpha,\beta$, $\Ult(V^1,\Uo{\alpha}{\beta}) \cong Z^0_{\alpha,\beta}[\Gone_{\alpha,\beta}] = 
Z_{\alpha,\beta}[\Gzero_{\alpha,\beta}*\Gone_{\alpha,\beta}]$,
where $Z^0_{\alpha,\beta}$ results from the iterated ultrapower $\T^0$.
Furthermore, $\Gzero = \Gzero_{\alpha,\beta}\uhr (\kappa+1)$ and
$\Gone = \Gone_{\alpha,\beta}\uhr \kappa$.\\
By Corollary \ref{Corollary - I - mo of U0} we have that
$\Uz{\alpha'}{\beta'} \mo \Uz{\alpha}{\beta}$ whenever $\alpha' <\alpha$. 
Also, by Definition \ref{Definition - I - pi0  alpha beta} we know that
the embedding $\pi^0_{\alpha,\beta}$ factors into  $\sigma^0_{1,\theta} \circ \sigma^0_{0,1}$, where
\[\sigma^0_{0,1} = 
\begin{cases} 
\jz{\beta}{\alpha} &\mbox{if }  \beta \leq \alpha \\
k^0_{\alpha,\beta} &\mbox{if } \alpha \geq \beta,
\end{cases}\]
 and that $\cp(\sigma^0_{1,\theta}) > \kappa$. 
 Therefore if $\Uz{\alpha'}{\beta'} \in \Mz{\alpha}{\beta}$ then 
 $\Uz{\alpha'}{\beta'} \in Z^0_{\alpha,\beta}$.
 We can now conclude the desired result by a simple case-by-case inspection:\\
\textbf{1.} When $\alpha' = \beta'$ and $\alpha = \beta$, we get that 
\[ \beta' < \alpha \Longrightarrow \Uz{\beta'}{\beta'} \in M^0_{(\alpha,\alpha)} \Longrightarrow   \Uz{\beta'}{\beta'} \in Z^0_{\alpha,\alpha}.\]
\textbf{2.} When $\alpha' = \beta'$ and $\alpha < \beta$, we have
\[ \beta' < \alpha \Longrightarrow \Uz{\beta'}{\beta'} \in M^0_{(\alpha,\beta)} \Longrightarrow   \Uz{\beta'}{\beta'} \in Z^0_{\alpha,\alpha}.\]
\textbf{3.} When $\alpha' < \beta'$ and $\alpha = \beta$, then
\[ \beta' < \alpha \Longrightarrow \Uz{\alpha'}{\beta'},\Uz{\beta'}{\alpha'} \in M^0_{(\alpha,\alpha)} \Longrightarrow   \Uz{\alpha'}{\beta'},\Uz{\beta'}{\alpha'} \in Z^0_{\alpha,\alpha}.\]
\textbf{4.} Finally, when $\alpha' < \beta'$ and $\alpha < \beta$, then
\[ \beta' < \alpha \Longrightarrow \Uz{\alpha'}{\beta'},\Uz{\beta'}{\alpha'} \in M^0_{(\alpha,\beta)} \Longrightarrow   Uz{\alpha'}{\beta'},\Uz{\beta'}{\alpha'} \in Z^0_{\alpha,\alpha}.\]
\end{proof}

This concludes the ``if" part of the proof. Before proceeding to the second part, let us first list several corollaries of inner model theory 
(\cite{zeman}). 

Suppose that $\Uo{\alpha'}{\beta'} \in M^1_{\alpha,\beta}$,
let  $j'_{\alpha',\beta'} : M^1_{\alpha,\beta} \to M'_{\alpha',\beta'} \cong \Ult(M^1_{\alpha,\beta},\Uo{\alpha'}{\beta'})$ be the ultrapower embedding.\\
\textbf{(a)}
\begin{enumerate}
 \item By the uniqueness of the core model, $\K(M^1_{\alpha,\beta}) = Z_{\alpha,\beta}$ ($Z_{\alpha,\beta}$ is described in Corollary \ref{Corollary - I - Restriction of pi0 to V}). 
 \item The restriction $\pi'_{\alpha',\beta'} = j'_{\alpha',\beta'}\uhr Z_{\alpha,\beta}$ can be realized as limit of a normal iteration 
 $T'$ of $Z_{\alpha,\beta}$.
 \item Let $G = \Gzero_{\alpha,\beta} * \Gone_{\alpha,\beta}$, then
\begin{enumerate}
 \item $G' = j'_{\alpha',\beta'}(G) \subset \pi'_{\alpha',\beta'}(\pzero * \pone)$
is $\pi'_{\alpha',\beta'}(\pzero * \pone)$ generic over $Z'_{\alpha',\beta'}$, 
\item $M'_{\alpha',\beta'} = Z'_{\alpha',\beta'}[G']$, and
\item for every $x \in \Mo{\alpha}{\beta} = Z_{\alpha,\beta}[G]$, if $x = (\name{x})_G$ then
$j'_{\alpha',\beta'}(x) = \pi'_{\alpha',\beta'}(\name{x})_{G'}$.
\end{enumerate}
\end{enumerate}
\noindent \textbf{(b) } 
The models $M'_{\alpha',\beta'} \cong \Ult(M^1_{\alpha,\beta},\Uo{\alpha'}{\beta'})$ and $\Mo{\alpha'}{\beta'} \cong \Ult(V^1,\Uo{\alpha'}{\beta'})$ have the same 
initial segment of the cumulative hierarchy, $V_{(j^1_{\alpha',\beta'}(\kappa))}$.
Indeedm, $\Mo{\alpha}{\beta} \cap (V^1)_{\kappa+1} = (V^1)_{\kappa+1}$ because $\Mo{\alpha}{\beta}$ is an ultrapower of $V^1$ by a $\kappa$ complete ultrafilter
Therefore when applying a $\Uo{\alpha'}{\beta'}$ ultrapower to $V^1$ and $\Mo{\alpha}{\beta}$, we find that 
$j^1_{\alpha',\beta'}\uhr (\kappa^+ + 1) = j'_{\alpha',\beta'}\uhr(\kappa^+ + 1)$ and 
$M'_{\alpha',\beta'} \cap V_{j^1_{\alpha',\beta'}(\kappa)} = \Mo{\alpha'}{\beta'} \cap V_{j^1_{\alpha',\beta'}(\kappa)}$.
It follows that 
\begin{enumerate}
 \item $\pi'_{\alpha',\beta'} \uhr  \kappa^+  = \pi^1_{\alpha',\beta'}\uhr \kappa^+$,
 \item $Z'_{\alpha',\beta'} \uhr {j^1_{\alpha',\beta'}(\kappa)} = Z_{\alpha',\beta'}\uhr{j^1_{\alpha',\beta'}(\kappa)}$, and
 \item The following normal iterations agree up to $j_{\alpha',\beta'}(\kappa^+) = \pi_{\alpha',\beta'}(\kappa^+)$:
 \begin{itemize}
  \item $T$ which generates  $\pi_{\alpha',\beta'} : V \to M_{\alpha',\beta'}$, and 
  \item $T'$, generating  $\pi'_{\alpha',\beta'} : M_{\alpha,\beta} \to M'_{\alpha',\beta'}$.
 \end{itemize}
 
  \end{enumerate}

\begin{claim}
If $\Uo{\alpha'}{\beta'} \mo \Uo{\alpha}{\beta}$  then $\beta' < \alpha$. 
\end{claim}
\begin{proof}
{We use the notations and results listed above.}
Since $Z'_{\alpha',\beta'}$ and $Z_{\alpha',\beta'}$ agree up to $\jo{\alpha'}{\beta'}(\kappa)$ we get that 
$o^{Z_{\alpha',\beta'}}(\nu) = o^{Z'_{\alpha',\beta'}}(\nu)$ for every $\nu < \jo{\alpha'}{\beta'}(\kappa)$. 

The first step of the iteration $T'$ coincides with the first step of the iteration
$T$. According to Corollary \ref{Corollary - I - Restriction of pi0 to V}, the first step of $T$ is an ultrapower by $U_{\alpha'}$;
thus $U_{\alpha'} \in Z_{\alpha,\beta}$. Since $o^{Z_{\alpha,\beta}}(\kappa) = \alpha$ it follows that
$\alpha' < \alpha$. Therefore if $\alpha' = \beta'$ then $\beta' < \alpha$, {as desired}. \\
Suppose now that $\alpha' < \beta'$. 
Since $\pi'_{\alpha',\beta'}$ is the embedding generated by $T'$, it factors into $\pi'_{\alpha',\beta'} = k' \circ j_{\alpha'}$,
where $j_{\alpha'} : Z_{\alpha,\beta} \to N'$ and $k' : N' \to Z'_{\alpha',\beta'}$, with $\cp(k') > \kappa$.
We have that $j_{\alpha',\beta'}(\kappa) > j_{\alpha'}(\kappa)$, therefore the iterations $T$ and $T'$ agree at $j_{\alpha'}(\kappa)$.
We also know $j_{\alpha'}(\kappa)$ is a criticaL point in $T$ via the ultrapower by $U = j_{\alpha'}(U_{\beta'})$. Note that 
$o(U) = j_{\alpha'}(\beta')$. Therefore the same holds for $T'$, and we must have that $U \in N'$. 
It follows that $j_{\alpha'}(\beta') = o(U) < o^{N'}(j_{\alpha'}(\kappa))$. Finally, $o^{Z_{\alpha,\beta}}(\kappa) = \alpha$ and we get that
$o^{N'}(j_{\alpha'}(\kappa)) = j_{\alpha'}(o^{Z_{\alpha,\beta}}(\kappa)) = j_{\alpha'}(\alpha)$.
We conclude that $j_{\alpha'}(\beta') < j_{\alpha'}(\alpha)$, therefore $\beta' < \alpha$. 
\end{proof}

\section{The Normal Measures on $\kappa$ in $V^1$}\label{Section - I - identify normal measures}

\begin{proposition}\label{Proposition - I - Uniquness of U^*}
The measures $\Uo{\alpha}{\beta}$, $\alpha,\beta < \lambda$ are the only measures on $\kappa$
in $V^1$.
\end{proposition}
\begin{proof}
Let $W$ be a normal measure on $\kappa$ in $V^1$, and $j_W : V^1 \to M_W \cong \Ult(V^1,W)$. 
There is a normal iteration $T^W$ of $V$ such that the resulting embedding $\pi : V \to M$ coincides with 
$j_W\uhr V$.
Moreover, if $V^1 = V[G]$ and $G = G^0 * G^1 \subset \pzero*\pone$, then $j_W(G) = j_W(G^0) * j_W(G^1)$ is $\pi(\pzero*\pone)$
generic over $M_W$. Denote $j_W(G^0)$, $j_W(G^1)$ by $G^0_W$, $G^1_W$ respectively. 
For every $M_W-$inaccessible  $\tau < j_W(\kappa)$ let $s^{G^0_W}_\tau$ be the $G^0_W-$induced generic Sacks function at $\tau$. 

According to Friedman-Magidor (\cite{Friedman-Magidor - Normal Measures}), $\pi``G^0$ determines the values of every $G^0_W$ Sacks function, $s^{G^0_W}_\tau$, 
with the exception of the values $s^{G^0_W}_{\gamma^1}(\gamma)$, where $\gamma$ if a critical point in $T_W$ and $\gamma^1$ is its 
 image\footnote{Namely, if $\gamma = cp(\pi_{i,i+1})$ is the critical point of the $i-$th stage of $T_W$ then 
 $\gamma^1= \pi_{i,i+1}(\gamma)$.}. 
In particular, $\kappa = \cp(\pi)$ and $\pi$ factors into $\pi = k \circ j_\beta$, where $\beta < o(\kappa)$ and  $\cp(k) > \kappa$.
Let $T_W^0$ be the lift of the iteration $T_W$ to $V^0 = V[G^0]$, and determined by $G^0_W$, and let $\pi^0 : V[G^0] \to M_W[G^0_W]$ be its induced embedding.

Let $\gamma  = s^{G^0_W}_{j_\beta(\kappa)}(\kappa)$, and 
$\Gzero_{\Uz{\beta}{\gamma}} = \jz{\beta}{\gamma}(\Gzero)$ be the $j_\beta(\pone)-$generic filter over $M_\beta$, 
associated with $\Uz{\beta}{\gamma}$. 
It follows that $\pi^0 = k^0 \circ \jz{\beta}{\gamma}$, where 
$k^0 : M_\beta[\Gzero_{\Uz{\beta}{\gamma}}] \to M_W[G^0_W]$ is
an extension of $k$. \\ 

\noindent \textbf{subclaim 1: $\gamma \geq \beta$.}\\
$o^{M_\beta}(\kappa) = \beta$ and $s^{G_W}_{j_W(\kappa)}(\kappa) = \gamma$, therefore
$\kappa \in j_W(\delset_\beta(\gamma))$.
If $\gamma$ was smaller than $\beta$, we would get that $\delset_\beta(\gamma) \subset \delset'$, 
i.e., $W$ concentrates on the set of non trivial iteration stages in $\pone$. 
Yet this contradicts the normality of $W$, as the generic Prikry function $d: \delset' \to \kappa$ is 
regressive and injective outside a bounded set.\\

\noindent {Recall that we have defined $\Gamma$ to be the set of all generic Prikry points, i.e., $\Gamma = \rng(d) = d``\delset'$.}\\
\noindent\textbf{subclaim 2: If $\Gamma \not\in W$ then $W = \Uo{\gamma}{\beta}$.}\\
It is sufficient to show that $\Uo{\gamma}{\beta} \subset W$. 
Suppose that $X \in \Uo{\gamma}{\beta}$, and let $\name{X}$ be a $\Gone$ name for $X$ in $V^0$.
According to remark \ref{Remark p- for U^* single ultrapower}  there is a condition
$p \in \Gone$ so that 
$\jz{\beta}{\gamma}(p)^{-\kappa} \force \can{\kappa} \in \jz{\beta}{\gamma}(\name{X}).$
By applying $k^0$ we get that
$\pi^0_W(p)^{-\kappa} \force \can{\kappa} \in \pi^0(\name{X})$.
Let $\Sigma \subset \kappa$ be the set of closure points of $d^{-1}$ (namely
$\nu \in \Sigma$ if and only if $d^{-1}(\nu) \subset \nu$). Using the Magidor iteration support, it is not difficult to verify that $\Sigma$ is closed unbounded in $\kappa$ (also, see \cite{OBN - Magidor Iteration}). Since $\Gamma \not\in W$, it follows that the set $\{ \nu < \kappa \mid p^{-\nu} \in \Gone\}$ belongs to $W$, thus $\pi^0(p)^{-\kappa} \in G^1_W$ and $X \in W$. \\

\noindent\textbf{subclaim 3: If $\Gamma \in W$ then $\gamma < \beta$ and $W = \Uo{\beta}{\gamma}$.}\\
Suppose now that $\Gamma \in W$.
Let $\Gamma' = \{ \alpha < \kappa : |d^{-1}(\alpha)|=1\}$.
It is not difficult to verify that $\Gamma \setminus \Gamma'$ is bounded in $\kappa$\footnote{also see \cite{OBN - Magidor Iteration}.}. 
Therefore if $\Gamma \in W$ then there exists a unique $\mu < j_W(\kappa)$ such that $j_W(d)(\mu) = \kappa$. \\
According to the results in \cite{OBN - Magidor Iteration} \footnote{i.e., the proofs of Proposition 3.2 and Lemma 3.6},
there is a finite subiteration of $T_W$, by which $\pi = k \circ j_\beta$ factors into
$\pi = e \circ j_{U'} \circ j_\beta$ so that
\begin{enumerate}
\item $j_{U'}$ is an ultrapower embedding by a normal measure $U'$ on $j_\beta(\kappa)$.
\item $U' = j_\beta(U_{\beta'})$ for some $\beta' < o(\kappa)$.
\item $\mu = e(j_\beta(\kappa))$.
\item $\cp(e) > \kappa$. 
\end{enumerate}
Let $\pi^0 = e^0 \circ j_{U^0} \circ \jz{\beta}{\gamma}$ be the corresponding factorization of the extension $\pi^0$ of $\pi$. 
In particular, $U^0 \in \Mz{\beta}{\gamma}$ extends $U' = j_\beta(U_{\beta'}) \in M_\beta$, and $\mu = e^0 \circ \jz{\beta}{\gamma}(\kappa)$.  
\\
We have $\kappa \in j_W(\Delta_\beta(\gamma))$.
If follows from the Definition of
$\pone$ that $\beta < \gamma$ and that $\mu = j_W(d^-1)(\kappa) \in \pi^0(\Delta_\gamma(\beta))$, i.e., 
$e^0 \circ \jz{\beta}{\gamma}(\kappa) \in e^0 \circ j_{U^0} \circ \jz{\beta}{\gamma}(\Delta_\gamma(\beta))$. 
{Therefore, it must means that  $U^0 = \jz{\beta}{\gamma}(U_{\gamma}(\beta))$} so we can rewrite $\pi^0$ as $\pi^0 = e^0 \circ k^0_{\beta,\gamma}$ (i.e., $k^0_{\beta,\gamma}$ in Definition \ref{Definition - I - U^* iterated ultrapower}).\\
According to remark \ref{Remark p+- for U^* iterated ultrapower} there is a condition
$p \in \Gone$ so that 
\begin{equation}\label{Equation - p+- iterated}
k^0_{\beta,\gamma}(p)^{+\left( \kappa, \jz{\beta}{\gamma}(\kappa) \right)-\kappa- \jz{\beta}{\gamma}(\kappa)}
		  \force \check{\kappa}\in k^0_{\beta,\gamma}(\name{X}).
\end{equation}
Let $\Pi = \{\nu \in \Gamma'  \mid \text{ for every } \mu < \kappa \text{ if } \mu> d^{-1}(\nu) \text{ then }(d(\mu)\not \in [\nu,d^{-1}(\nu)])\ \}$.
It is not difficult to verify that $\Gamma \setminus \Pi$ is bounded in $\kappa$ (see \cite{OBN - Magidor Iteration}), and that for every
$p \in \Gone$ and $\mu < \kappa$, 
$p^{(+\mu,d^{-1}(\mu)) - \mu - d^{-1}(\mu)} \in \Gone$ whenever 
$\mu \in \Pi \cap \Sigma$.\\
Since $\Gamma \in W$ and $\Sigma \subset \kappa$ is a club, it follows 
that $\Pi \cap \Sigma \in W$.
We conclude that $\pi^0(p)^{+\left( \kappa, \mu \right)-\kappa- \mu} \in G^1_W$.
By Applying $e^0$ to equation \ref{Equation - p+- iterated} we conclude 
that 
\[\pi^0(p)^{+\left( \kappa, \mu \right)-\kappa- \mu}
		  \force \check{\kappa}\in \pi^0(\name{X}).\]
Therefore $X \in W$ 
\end{proof}

\section{A Final Cut}\label{Section - I - final cut}

According to Friedman and Magidor (\cite{Friedman-Magidor - Normal Measures}),
there is a sequence $\vec{X}^\kappa = \la X^\kappa_i \mid i < \kappa^+\ra$
of pairwise disjoint stationary subsets of $\kappa^+ \cap \Cf(\kappa)$ in $V^1$ and a function $f : \kappa \to V^1$, so that $j(f)(\kappa) = \vec{X}^\kappa$ for
every elementary embedding $j$ in $V^1$ with $\cp(j) = \kappa$\footnote{i.e., we can use a $\Diamond_{\kappa^+}$ sequence in $V = \K(V^1)$ which is definable from $H(\kappa^+)^V$.}. 
We may assume that $f(\nu) = \la X^\nu_i \mid i < \nu^+\ra$ is a $\nu^+-$sequence of disjoint stationary subsets of $\nu^+ \cap \Cf(\nu)$ for each $\nu < \kappa$, 

\begin{definition}[$\Code^*(\nu)$, $\nu < \kappa$]\label{Definition - code X}
A condition in $\Code^*(\nu)$ is a closed, bounded subset $c$ of $\nu^+$ which are disjoint from $X^\nu_0$.
For conditions $c,d \in \Code^*(\nu)$, $d \geq c$, if and only if:
\begin{enumerate}
\item $d$ end extends $c$.
\item For $i \leq \max(c)$: if $i$ belongs to $c$ then $d \setminus c$ is disjoint from $X^\nu_{1+2i}$; if $i$ does not belong to $c$ then $d \setminus c$ is disjoint from $X^\nu_{1+2i + 1}$.
\end{enumerate}
\end{definition}
For a set $X \subset \kappa$ in $V^1$, let $\pX$ be {a variation} of the Friedman-Magidor iteration, 
$\pX = \pX_{\kappa} = \la \pX_\nu,\qX_\nu \mid \nu < \kappa\ra$, where
\[
\qX_\nu = 
\begin{cases}
    \Code^*(\nu) &\mbox{if }  \nu \in X \text{ is inaccessible.} \\
    \text{The trivial poset}  &\mbox{otherwise}.
\end{cases}
\]
\begin{lemma}\label{Lemma - I - Cut Down Measure Preservation}
Let $X \subset \kappa$ be a set in $V^1$, and  $\GX \subset \pX$ be a generic filter over $V^1$. 
For every normal measure $U$ on $\kappa$ in $V^1$, if $X \not\in U$ then $U$ has a unique extension $U^X$ in  $V^1[\GX]$. Furthermore, these are the only normal measure on $\kappa$ in $V^1[\GX]$.
\end{lemma}
\begin{proof}
Let $U \in V^1$ be  normal measure on $\kappa$ such that $X \not\in U$,
and $j : V^1 \to M^1 \cong \Ult(V^1,U)$ be its ultrapower embedding.
We have $j(\pX)\uhr\kappa = \pX$. Also, stage $\kappa$ in 
$j^1(\pX)$ is trivial as $\kappa \not\in j(X)$. 
Like the Friedman-Magidor poset, $\pX$ satisfies that for every dense open set $D \subset j(\pX)$, there is some $g \in \GX$ so that
$j(g)$ reduces $D$ to a dense open set in $j(\pX)\uhr(\kappa+1) = \pX$, which is intersected by $\GX$. Thus $j``\GX$ determines a unique generic filter $H^X \subset j(\pX)\setminus \kappa$ over $M^1[\GX]$. Setting $G^* = \GX * H^X$, we get that $G^* \subset j(\pX)$ is the unique generic filter
over $M^1$ for which $j``\GX \subset G^*$.  
It follows that
$j^* : V^1[\GX] \to M^1[G^*]$ is the only extension of $j : V^1 \to M^1$ to $V^1[\GX]$ and that $U^X = \{Y \subset \kappa \mid \kappa \in j^*(Y)\}$ is the only extension of $U$ in $V^1[\GX]$.
For $\alpha,\beta < o(\kappa)$ such that  $X \not\in \Uo{\alpha}{\beta}$, we denote $(\Uo{\alpha}{\beta})^X$ by $\UX{\alpha}{\beta}$.

Suppose now that $W \in V^1[\GX]$ is a normal measure on $\kappa$ and $j_W : V^1[\GX] \to M_W$ be the resulting ultrapower embedding.
Then $j = j_W \uhr V : V \to M$ is an iteration of $V$ and
$G_W = j_W(G) \subset j(\pzero * \pone * \pX)$ is generic over $M$. 
We first claim $X \not\in W$. Otherwise, $\kappa \in j_W(X)$, so $\kappa$ is a non trivial forcing stage
in $j_W(\pX)$, and $\qX_\kappa = \Code^*(\kappa)$. 
It follows that $G_W$ introduces a club $D \subset \kappa^+$, disjoint 
from $j_W(f)(\kappa)_0 = X^\kappa_0$.
Note that $D$ is a club in $V^1[\GX]$ since $M_W$ is closed under $\kappa-$sequences. This is absurd as $X^\kappa_0$ is stationary in $V^1$ and $|\pX| = \kappa$.\\
To show that $W = \UX{\alpha}{\beta}$ for some $\alpha,\beta < o(\kappa) = \lambda$, it is sufficient to verify that $\Uo{\alpha}{\beta} \subset W$.
This is an immediately consequence of the proof of Proposition \ref{Proposition - I - Uniquness of U^*}: Considering the restriction $\pi = j_W\uhr V  : V \to M_W$ and its extension $\ref{Proposition - I - Uniquness of U^*}= j_W\uhr V^0 : V[\Gzero] \to M_W[\Gzero_W]$, 
we get that $\alpha,\beta$ are determined from the values $o^{M_W}(\kappa)$, 
$s^{G^W}_{j_W(\kappa)}(\kappa)$, and whether $\Gamma \in W$.
{The proof of Proposition \ref{Proposition - I - Uniquness of U^*} relies solely on the analysis of the iterations of $\pi$, $\pi^0$, and therefore applies here is well.}
\end{proof}
\noindent Suppose $\UX{\alpha}{\beta} \in V^1[\GX]$ be a normal measure on $\kappa$, and  let $\jX{\alpha}{\beta} : V^1[\GX] \to \MX{\alpha}{\beta}$ be its ultrapower embedding.
We have that $\jX{\alpha}{\beta}\uhr V^1 = \jo{\alpha}{\beta}$, thus
$\jX{\alpha}{\beta}\uhr V^0 = \pi^0_{\alpha,\beta}: V^0 \to Z^0_{\alpha,\beta}$. $\pi^0_{\alpha,\beta}$, $Z^0_{\alpha,\beta}$ were used to determine the Mitchell order on $\Uo{\alpha}{\beta}$, in Proposition \ref{Proposition - MO structure on U1}. 
It follows that the proof of this Proposition applies to $\UX{\alpha}{\beta}$ as well.
\begin{corollary}\label{Lemma - I - Cut down Lemma}
Suppose that $\UX{\alpha'}{\beta'},\UX{\alpha}{\beta} \in V^1$ where $\alpha' \leq \beta'$ and $\alpha \leq \beta$. We have that $\UX{\alpha'}{\beta'} \mo \UX{\alpha}{\beta}$ if and only if
$\beta' < \alpha$.
\end{corollary}

\begin{lemma}[The final cut]\label{Lemma - The final cut}
Let $\kappa$ be a measurable cardinal in a transitive model of set theory $V$
so that the normal measures on $\kappa$ are separated by sets.
Suppose that for every $X \subset \kappa$ there is a poset $\pX \in V$ so that
\begin{enumerate}
\item The normal measures on $\kappa$ which extend in a $\pX$ generic extension, are exactly the normal measures $U \in V$ which do not contain $X$. Furthermore, If $X \not\in U$ then $U$ has a unique extension $U^X \in V^{\pX}$.

\item $\pX$ preserves the Mitchell order in $V^1$. Namely, for every $U,W \in V$ which extend to $U^X,W^X$ respectively, 
$U^X \mo W^X$ if and only if $U \mo W$. 
\end{enumerate}
Then for every $\W \subset \mo(\kappa)^V$ of cardinality $\leq \kappa$
there is a set $X \subset \kappa$ such that 
$\mo(\kappa)^{V^{\pX}} \cong \mo(\kappa)^V \uhr \W$.  
\end{lemma}

\begin{proof}
Let $\la U_i \mid i < \rho\ra$ be an enumeration of $\W$, where $\rho \leq \kappa$ is a cardinal.
For every $i  <\rho$ let $X_i \subset \kappa$ be a set which separates $U_i$
from the rest of the normal measures on $\kappa$. Let $X_{\W} = \triangle_{i < \rho}X_i$, where $\triangle_{i<\rho}$ is the diagonal union if $\rho = \kappa$, and an ordinary union otherwise. It follows that the set $X = \kappa \setminus X_{\W}$ belongs to a normal measure $U \in V$, if and only if $U \not\in \W$. 
Thus, it follows from the rest of the assumptions that
$\mo(\kappa)^{V^{\pX}} \cong \mo(\kappa)^V \uhr \W$.
\end{proof}

\begin{theorem}\label{Theorem - I - main theorem}
Suppose that $V = L[\U]$ is a Mitchell model and $(S,<_S)$ is a tame order so that
$|S| \leq \kappa$ and $\tamerank(S) \leq o^{\U}(\kappa)$. Then there is a cofinality preserving generic extension $V^*$ of $V$ such that $\mo(\kappa)^{V^*} \cong \thinspace (S,<_S)$.
\end{theorem}

\begin{proof}
If $\lambda \leq \kappa$ and $(S,<_S)$ is reduced then $(S,<_S)$ embeds
into $R_\lambda$ for every $\lambda \geq \tamerank(S,<_S)$ (Proposition \ref{Proposition - tame}). We verify that the claim is an immediate consequence of the results established in Sections \ref{Section - I - mo structure} and \ref{Section - I - identify normal measures}. We may assume that $S \subset R_\lambda$, and force with $\pzero*\pone$ over $V = L[\U]$ to obtain a generic extension $V^1$ of $V$, so that
\begin{itemize}
\item the normal measures on $\kappa$ are separated by sets (Proposition \ref{Proposition - I - Uniquness of U^*} and Corollary \ref{Corollary - I - U1 separating sets}), and
\item there are distinguished normal measures $\Uo{\alpha}{\beta}$, $\alpha \leq \beta < \lambda$,
so that $\mo(\kappa)^{V^1}\uhr \{\Uo{\alpha}{\beta} \mid \alpha \leq \beta\} \thinspace \cong \thinspace R_\lambda$ (Proposition \ref{Proposition - MO structure on U1}).
\end{itemize} 
Let $\W = \{\Uo{\alpha}{\beta} \mid (\alpha,\beta) \in S\}$. $|\W| \leq \kappa$ since $|S| \leq \kappa$, and by Lemma \ref{Lemma - The final cut}, there is
a set $X \subset \kappa$ so that in a generic extension of $V^1$ by $\pX$, $\mo(\kappa) \cong \mo(\kappa)^{V^1}\uhr S \cong (S,<_S)$.\\

\noindent Next, we describe how to modify $\pzero$ and $\pone$ to deal with arbitrary tame orders $(S,<_S)$ of cardinality $\leq \kappa$. Let $\la \rho_\tau \mid \tau < \kappa^+\ra$ be a sequence of canonical functions on $\kappa$,
so that each $\rho_\tau$ has Galvin-Hajnal norm $\tau$. 
If $j : V \to M$ with $\cp(j) = \kappa$ then $j(\rho_\tau)(\kappa) = \tau$ for every $\tau < \kappa^+$. Also, for every $\alpha<\beta< \kappa^+$ the set $\{\nu < \kappa \mid \rho_\alpha(\nu) \not < \rho_\beta(\nu)\}$ is bounded in $\kappa$.
Since $\lambda < \kappa^+$ we may choose the functions $\la \rho_\tau \mid \tau < \lambda\ra$ so that
 $\{\nu < \kappa \mid \rho_\alpha(\nu) \geq \rho_\beta(\nu)\} = \emptyset$ for every
$\alpha < \beta \leq \lambda$. 
For each $\alpha < \lambda$ let $\Delta_\alpha =  \{\nu < \kappa \mid o(\nu) = \rho_{\alpha}(\nu)\}$. It follows that the sets $\Delta_\alpha$, $\alpha < \lambda$, are pairwise disjoint. 
We proceed as follows:\\

\noindent\textbf{1.} It is not difficult to verify that there is a set $\Delta \in \bigcap_{\alpha < \lambda}U_{\kappa,\alpha}$, so that each $\nu \in \Delta$ is an inaccessible cardinal, a closure point of $\rho_\lambda$, and satisfies that $\rho_\lambda\uhr \nu$ has a Galvin-Hajanl rank $\rho_\lambda(\nu) < \nu^+$.
$\pzero = \la \pzero_\nu, \qzero_\nu \mid \nu \leq \kappa\ra$ is a Friedman-Magidor iteration where
for each $\nu < \kappa$,  $\qzero_\nu$ is non-trivial if and only if $\nu \in \Delta \cup\{\kappa\}$, where $\qzero_\nu = \Sacks_{\rho_\lambda\uhr \nu}(\nu) * \Code(\nu)$ is defined by
\begin{itemize}
\item conditions $T \in \Sacks_{\rho_\lambda\uhr \nu}(\nu)$ are the
 trees $T \subset {}^{<\nu}\nu\times\nu$ for which there is a club $C \subset \nu$ so that if 
$s \in T$ and $\len(s) \in C$ then $s \fr \la ( \eta,\mu ) \ra \in T$ for every $\eta < \rho_{\lambda}(\len(s))$ and $\mu < \len(s)$.\\
The forcing $\Sacks_{\rho_\lambda\uhr \nu}(\nu)$ introduces a generalized Sacks function 
$s_\nu : \nu \to \rho_{\lambda}(\nu) \times \nu$.
\item $\Code(\nu)$ is a Friedman-Magidor coding poset, which introduces a club $C_\nu \subset \nu^+$ coding both $s_\nu$ and itself.
\end{itemize}

\noindent Let $V^0 = V[\Gzero]$ where $\Gzero \subset \pzero$ is a $V-$generic filter. 
For each $( \eta,\mu)\in \lambda \times \kappa$ and $\alpha < o(\kappa)$ define
$\Delta_\alpha(\eta,\mu) = \{ \nu \in \Delta \cap \Delta_\alpha \mid s_\nu = s_\kappa\uhr\nu \text{ and } s_\kappa(\nu) = (\rho_\eta(\nu),\mu)\}$.
We get that $\{ \Delta_\alpha(\eta,\mu) \mid \alpha < o(\kappa),\eta < \lambda,\mu < \kappa\}$ are pairwise disjoint. 
The description of the normal measures on $\kappa$ in Section \ref{Section - I - Posets} show that each  normal measure $U_{\alpha}$ in $V$ extends in $V^0$ to $\{\Uz{\alpha}{\eta,\mu} \mid \eta < \lambda, \mu < \kappa\}$ and that $\Delta(\eta,\mu) \in \Uz{\alpha}{\eta,\mu}$. \\
The parameters $\alpha,\eta < \lambda$ in $\Uz{\alpha}{\eta,\mu}$ will be associated with elements $(\alpha,\eta) \in R_\lambda$. The additional parameter $\mu < \kappa$ 
will guarantee that there are $\kappa$ $\mo-$equivalent copies of each $(\alpha,\eta) \in R_\lambda$, thus allowing us to realize non-reduced orders $(S,<_S)$ where each $\sim_S$ equivalent class has cardinality $\leq \kappa$. \\

\noindent \textbf{2.}
Next, we force over $V^0$ with a Magidor iteration of Prikry forcings,
$\pone = \la \pone_\nu, \qone_\nu \mid \nu < \kappa\ra$.
The recipe for choosing the normal measure on $\nu$ to be used at non-trivial iteration stages, is similar to the recipe used in Section \ref{Section - I - Posets} (Definition 
\ref{Definition - recipe for Qone}), i.e., if $\nu \in \Delta_\alpha(\beta,\mu)$ for 
some $\beta < \lambda$ and $\mu < \kappa$, then
\[\qone_\nu = 
\begin{cases}
    Q(U^1_{\nu,(\alpha,\beta,\mu)}) &\mbox{if }  \beta < \alpha \\
  0 - \text{the trivial forcing}  &\mbox{otherwise}
\end{cases}\]

\noindent Here, $U^1_{\nu,(\alpha,\beta,\mu)}$ is a normal measure on $\nu$ in $V^0[G^1\uhr \nu]$
which extends the measure $U^0_{\nu,(\beta,\alpha,\mu)} \in V^0$ (thus, extending $U_{\nu,\beta} \in V$). The definitions of $U^1_{(\alpha,\beta,\mu)}$ ($\alpha \geq \beta$ and $\alpha < \beta$)  are similar to those of $U^1_{\alpha,\beta}$. Here, for $\alpha \geq \beta$, the $U^0_{\beta,\alpha}$ ultrapower in Definition \ref{Definition - I - U^* single ultrapower} is replaced with an ultrapower by 
$U^0_{\beta,\alpha,\mu}$; for $\alpha < \beta$, the ultrapower by $U^0_{\alpha,\beta} \times U^0_{\beta,\alpha}$ in Definition \ref{Definition - I - k alpha beta embedding} is replaced with an ultrapower by $U^0_{\alpha,\beta,\mu} \times U^0_{\beta,\alpha,\mu}$. \\
Therefore, a $V^0$ generic filter $\Gone \subset \pone$ introduces a Prikry (partial) function $d : \Delta \to \kappa$, where 
\begin{itemize}
\item $\nu \in \dom(d)$ if and only if there are $\alpha < \beta < \lambda$ and $\mu < \nu$ so that $\nu \in \Delta_\beta(\alpha,\mu)$, and then 
\item $d(\nu) \in \Delta_\alpha(\beta,\mu) \cap \nu$ (for all but finitely many $\nu$).
\end{itemize}
It follows that for every $\alpha < \beta <\lambda$ and $\mu < \kappa$, the function $\nu \mapsto (\nu,d^{-1}(\nu))$ introduces a projection of $\Uo{\alpha}{ \beta,\mu} \in V^1$ to an extension of the product
$\Uz{\alpha}{\beta,\mu} \times \Uz{\beta}{\alpha,\mu} \in V^0$ (thus, extending $U_\alpha \times U_\beta \in V$).
When $\alpha = \beta$, $\Uo{\alpha}{\alpha,\mu}$ extends $\Uz{\alpha}{\alpha,\mu} \in V^0$. 
The obvious modification of the proof of Proposition \ref{Proposition - MO structure on U1} implies that for every  $\Uo{\alpha'}{\beta',\mu'}$,$\Uo{\alpha}{\beta,\mu}$ in $V^1$, where $\alpha \leq \beta$ and $\alpha'\leq \beta'$, we have that 
\[\Uo{\alpha'}{\beta',\mu'} \mo \Uo{\alpha}{\beta,\mu} \iff \beta' < \alpha.\]
In particular, when restricting $\mo$ to these measures we see that for every
$\alpha \leq \beta < \lambda$, the normal measures in $\{ \Uo{\alpha}{\beta,\mu} \mid \mu < \kappa\}$ are $\mo$ equivalent.\\

\noindent \textbf{3.}
Let $([S],<_{[S]})$ be the reduction of $(S,<_S)$. $([S],<_{[S]})$ is reduced and $\tamerank([S],<_{[S]}) = \tamerank(S,<_S) = \lambda$. Proposition \ref{Proposition - tame} implies that $([S],<_{[S]})$ embeds in $(R_\lambda,<_{R_\lambda})$.
Since each equivalent class in $[S]$ has size at most $\kappa$, it follows that there is a subset $\W \in V^1$ of normal measures on $\kappa$,
such that $\mo(\kappa)^{V^1} \uhr \W \cong (S,<_S)$.  By Lemma \ref{Lemma - The final cut} there is a set $X \subset \kappa$ so that in a generic extension of $V^1$ by $\pX$, $\mo(\kappa) \cong \mo(\kappa)^{V^1}\uhr \W \cong (S,<_S)$.
\end{proof}

\noindent\textbf{Acknowledgements -} 
The author would like to express his gratitude to his supervisor Professor Gitik,
 for many fruitful conversations, valuable guidance and encouragement.
The author is also grateful to the referee for making valuable suggestions and comments which 
greatly improved both the content and structure of this paper.

%
%
\raggedright

\end{document}